\documentclass[11pt,a4paper]{amsart}
\usepackage{amsmath}%
\usepackage{amsfonts}%
\usepackage{amssymb}%
\usepackage{graphicx}
\usepackage{amsthm}
\usepackage{amssymb}
\usepackage{tikz}
\usepackage{tikz-cd}
\usetikzlibrary{matrix,arrows}
\usepackage[english, ngerman]{babel}
\usepackage[utf8]{inputenc}
\usepackage{hyperref}
\usepackage{cite}
\usepackage{footnote}
\usepackage{stmaryrd}
\usepackage[a4paper,margin=3cm]{geometry}
\usepackage{verbatim}

\newtheorem{thm}{Theorem}
\theoremstyle{plain}

\newtheorem{cor}[thm]{Corollary}

\newtheorem{lem}[thm]{Lemma}

\newtheorem{prop}[thm]{Proposition}

\theoremstyle{remark}
\newtheorem{rem}[thm]{Remark}
\newtheorem*{acknowledgements}{Acknowledgements}

\numberwithin{equation}{section}
\numberwithin{thm}{section}

\renewcommand{\phi}{\varphi}

\newcommand{\BN}{{\mathbb{N}}}

\newcommand{\BR}{{\mathbb{R}}}

\newcommand{\BZ}{{\mathbb{Z}}}

\newcommand{\Ff}{{\mathfrak{f}}}

\newcommand{\FB}{{\mathfrak{B}}}

\newcommand{\FF}{{\mathfrak{F}}}

\newcommand{\FS}{{\mathfrak{S}}}

\newcommand{\CC}{{\mathcal C}}

\newcommand{\CI}{{\mathcal I}}
\newcommand{\CJ}{{\mathcal J}}

\newcommand{\CM}{{\mathcal M}}
\newcommand{\CN}{{\mathcal N}}

\newcommand{\CS}{{\mathcal S}}

\newcommand{\CW}{{\mathcal W}}

\newcommand{\CZ}{{\mathcal Z}}

\newcommand{\vect}[1]{{\boldsymbol{#1}}}

\newcommand{\nin}{\notin}

\renewcommand{\mod}{\mathop{\rm mod}\nolimits}

\renewcommand{\binom}[2]{\left({#1}\atop{#2}\right)}

\newcommand{\quotient}[2]{
        \mathchoice
            {
                \text{\raise1ex\hbox{$#1$}\Big/\lower1ex\hbox{$#2$}}%
            }
            {
                #1\,/\,#2
            }
            {
                #1\,/\,#2
            }
            {
                #1\,/\,#2
            }
    }
		
\newcommand{\rquotient}[2]{
        \mathchoice
            {
                \text{\lower1ex\hbox{$#1$}\Big \backslash \raise01ex\hbox{$#2$}}%
            }
            {
                #1\,\backslash\,#2
            }
            {
                #1\,\backslash\,#2
            }
            {
                #1\,\backslash\,#2
            }
    }
		
\newcommand{\lrquotient}[3]{
        \mathchoice
            {
                \text{\lower1ex\hbox{$#1$}\Big \backslash \raise01ex\hbox{$#2$}\Big/\lower1ex\hbox{$#3$}}%
            }
            {
                #1\,\backslash\,#2\,/\,#3
            }
            {
                #1\,\backslash\,#2\,/\,#3
            }
            {
                #1\,\backslash\,#2\,/\,#3
            }
    }

\begin{document}
\selectlanguage{english}

\bibliographystyle{plain}

\title{Effective Vinogradov's Mean Value Theorem via Efficient Boxing}
\author{Raphael S. Steiner}
\address{Department of Mathematics, University of Bristol, Bristol BS8 1TW, UK}%
\email{raphael.steiner@bristol.ac.uk}%




\begin{abstract} We combine Wooley's efficient congruencing method with earlier work of Vinogradov and Hua to get effective bounds on Vinogradov's mean value theorem.
\end{abstract}


\maketitle
\tableofcontents

\section{Introduction}

In this paper we are concerned with Vinogradov's mean value theorem. Let $k$ and $s$ denote two natural numbers. The goal is to understand integer solutions of the Diophantine equation
\begin{equation}
\sum_{i=1}^s x_i^j = \sum_{i=1}^s y_i^j, \quad (j=1,\dots,k),
\label{eq:vinomean}
\end{equation}
with $0 < \vect{x},\vect{y} \le X$. By orthogonality the number of solutions is equal to
$$
J_{s,k}(X)=\int_{[0,1[^k}|f(X/2,X,\vect{\alpha})|^{2s}d\vect{\alpha},
$$
where we define
$$
f(N,M,\vect{\alpha})=\sum_{N-\frac{1}{2}M < x \le N+\frac{1}{2}M} e(\vect{\alpha} \cdot \vect{\vartheta}(x))
$$
for real $N,M$ with $M\ge 1$ and $\vect{\vartheta}(x)=(x,x^2,\dots,x^k)$. Lower bounds for $J_{s,k}(X)$ are well-known and easily proved (see for example \cite{Vaughanhardylittlewood}). They admit the form
\begin{equation}
J_{s,k}(X) \gg_{s,k} \max\{ X^s, X^{2s-\frac{1}{2}k(k+1)}\}.
\label{eq:lowerbound}
\end{equation}
In a recent breakthrough Bourgain, Demeter and Guth \cite{BourgainVMVT} have shown that \eqref{eq:lowerbound} is sharp up to a factor $X^{\epsilon}$; i.e. they have proven the inequality
\begin{equation}
J_{s,k}(X) \ll_{s,k,\epsilon} \max\{ X^{s+\epsilon}, X^{2s-\frac{1}{2}k(k+1)+\epsilon} \}
\label{eq:conj}
\end{equation}
to hold for all $s,k \in \BN$ and $\epsilon>0$. An application of the circle method then further shows that one has an asymptotic of the shape
\begin{equation}
J_{s,k}(X) \sim C_{s,k} X^{2s-\frac{1}{2}k(k+1)}
\label{eq:asym}
\end{equation}
for all $s>\frac{1}{2}k(k+1)$. Before this latest breakthrough there has been a long history of improvements towards \eqref{eq:conj} and \eqref{eq:asym}. Following Vinogradov \cite{vinogradov1935new}, who gave an estimate of the shape
\begin{equation*}
J_{s,k}(X) \ll_{s,k} X^{2s-\frac{1}{2}k(k+1)+\eta_{s,k}},
\end{equation*}
there have been improvements in the argument by Linnik \cite{LinnikVino}, Karatsuba \cite{karatsuba1973mean} and Stechkin \cite{stechkin1975mean} leading to an error in the exponent of only $\eta_{s,k}=\frac{1}{2}k^2(1-1/k)^{\lfloor s/k \rfloor}$. This allows one to get the asymptotic \eqref{eq:asym} as soon as $s \ge 3k^2(\log k+O(\log \log k))$. By extending Linnik's argument Wooley \cite{Wooley92} was further able to decrease the exponent to roughly $\eta_{s,k}=k^2e^{-2s/k^2}$ using his efficient differencing method, which allowed him to show that the asymptotic \eqref{eq:asym} holds for $s \ge k^2(\log k +O(\log \log k))$. Later Wooley \cite{WooleyAnnals} developed a powerful new argument, called efficient congruencing, which enabled him to prove \eqref{eq:conj} for $s \ge k(k+1)$. Note that this is just a factor of $2$ off the critical case $s=\frac{1}{2}k(k+1)$, from which all other cases would follow. There have followed a series of papers in which Wooley has refined his method, leading to proofs of \eqref{eq:conj} for $s\le \frac{1}{2}k(k+1)-\frac{1}{3}k+O(k^{\frac{2}{3}})$ \cite{Wooleybelow} and a full proof when $k=3$ \cite{Wooleycubic}. The history of the main conjecture \eqref{eq:conj} ends with Bourgain, Demeter and Guth's full proof using decoupling theory from harmonic analysis.\\

Vinogradov's mean value theorem has a broad range of applications. For example it can be used to get strong bounds on exponential sums (see Chapter 8.5 in \cite{ANT}). These strong bounds can then be used to get zero-free regions of the Riemann-Zeta function, something which has been made explicit by Ford \cite{FordAnnals}. Furthermore they have been used by Hal\'asz and Tur\'an \cite{HaTu69} to get zero-density estimates for the Riemann-Zeta function. Other applications of Vinogradov's mean value theorem include estimates for short mixed character sums, such as found in work of Heath-Brown and Pierce \cite{Heath-Brown01062015} and Kerr \cite{Kerr14}, as well as contributions to restriction theory worked out by Wooley \cite{WooleyRestriction} and Bourgain, Demeter and Guth \cite{BourgainVMVT}. In all of these applications it is desirable to have an effective version of Vinogradov's mean value theorem.

Effective versions have been given by Hua \cite{Huamean}, whose argument is based on Vinogradov's original method, Stechkin \cite{stechkin1975mean} as well as by Arkhipov, Chubarikov and Karatsuba \cite{ACK04}, whose work is based on Linnik's $p$-adic argument, and Ford \cite{FordAnnals}, whose argument is based on Wooley's efficient differencing method. In this paper we prove effective bounds using Wooley's efficient congruencing method combined with the older arguments of Vinogradov and Hua.

Let us give an overview of the heart of Vinogradov's and Hua's argument. If we have two tuples $\vect{x},\vect{x'}$ such that $\|\vect{x}-\vect{x'}\|_{\infty}\le S$, then we have
\begin{equation}
\left|\sum_{i=1}^k (x_i^j-x_i'^j) \right| \le j k \cdot S X^{j-1}.
\label{eq:small}
\end{equation}
The question of whether this can be reversed arises naturally and the answer is in the affirmative, although it depends on how well-spaced $\vect{x}$ is; i.e. how large $\min_{i\neq j} |x_i-x_j|$ is (see Lemma \ref{lem:smallint}). In his paper \cite{Huamean} Hua uses this reversability by writing \eqref{eq:vinomean} as
\begin{equation}
\sum_{i=1}^k (x_i^j-x_i'^j) =\sum_{i=1}^{s-k} (y_i^j-y_i'^j).
\label{eq:hua}
\end{equation}
By splitting up the $\vect{y}$ and $\vect{y'}$ into $X^{\frac{2(s-k)}{k}}$ intervals of length at most $X^{1-\frac{1}{k}}$ and using the integer translation invariance in combination with H\"older's inequality one can force $\|\vect{y}\|_{\infty},\|\vect{y'}\|_{\infty} \le X^{1-\frac{1}{k}}$ in \eqref{eq:hua}. Now the right hand side of \eqref{eq:hua} is small. It is in fact at most $(s-k) X^{j-1} X^{1-\frac{j}{k}}$. Splitting up the right hand side further into $(s-k)X^{1-\frac{j}{k}}$ intervals of size $X^{j-1}$ and using Cauchy-Schwarz one is able to reduce to \eqref{eq:small} with $S=1$. Fixing the $\vect{x'}$ arbitrarily allows now only $O_{s,k}(1)$ choices for the $\vect{x}$ as $x_i=x_i'+O_{s,k}(1)$ and the choices for $\vect{y}$ and $\vect{y'}$ can be bounded by $J_{s-k,k}(X^{1-\frac{1}{k}})$. This gives
$$
J_{s,k}(X) \ll_{s,k} \log(2X)^2 \cdot  X^{\frac{2(s-k)}{k}} \cdot \prod_{j=1}^k X^{1-\frac{j}{k}} \cdot X^k \cdot J_{s-k,k}(X^{1-\frac{1}{k}}),
$$
where the $\log(2X)^2$ is coming from a dyadic argument ensuring that $\min_{i\neq j} |x_i-x_j|$ is not too small. Iterating this inequality $l$-times, Hua proved the following upper bound for $s \ge \frac{1}{4}k(k+1)+lk$:
$$
J_{s,k}(X) \le (7s)^{4sl} \log(X)^{2l} X^{2s-\frac{1}{2}k(k+1)+\frac{1}{2}k(k+1)(1-\frac{1}{k})^l} \quad \forall X \ge 2.
$$

The same kind of argument also works if we only force $\|\vect{y}\|_{\infty},\|\vect{y'}\|_{\infty} \le X^{1-\theta}$, with $\theta$ very small. In this case one concludes $x_i=x_i'+O(X^{1-k\theta})$ and one can put the $\vect{x}$'s  into a box of size $X^{1-k\theta}$. Now Wooley's efficient machinery  lets us interchange the roles of $\vect{x}$ and $\vect{y}$ and thus allows us to play the same game again with $k\theta$ instead of $\theta$. In every iteration there is a slight gain in the exponent, depending on $s,k,\theta$ and $\eta_{s,k}$. In the simplest form of Wooley's efficient machinery these gains stack up to overcome the defect of the method as soon as $s \ge k(k+1)$ leading to a slight decrease of $\eta_{s,k}$ as seen in the following theorem.
\begin{thm}
Let $s,k \in \BN$ with $k \ge 3$ and $2 \log(k) \ge \lambda=\frac{s-k}{k^2}\ge 1$. Assume that
$$
J_{s,k}(X) \le C \log_2(2X)^{\delta} X^{2s-\frac{1}{2}k(k+1)+\eta} \quad \forall X \ge 1,
$$
for some $0 \le \delta$ and $0 < \eta \le \frac{1}{2}k(k+1)$. Further let
$$
D \ge \max \left\{1, \frac{\log\left(\frac{k^2}{2\eta} \frac{\lambda-1}{\lambda^2}+1 \right)}{\log(\lambda)} \right\}
$$
be an integer and set $\theta=k^{-(D+1)}$. Then we have
$$
\begin{aligned}
J_{s,k}(X) \le & C \cdot 2^{\frac{3}{2}k^2+\frac{11}{2}k+1}k^{\frac{1}{2}k^2+\frac{25}{6}k-2}  \cdot \CM_0\\
& \cdot \log_2(2X)^{\delta+ \frac{2 \lambda k -1 }{\lambda k -1}} X^{2s-\frac{1}{2}k(k+1)+\eta} \cdot X^{-\eta \theta \frac{s-2k}{s-k}}, \qquad \qquad \forall X \ge 1,
\end{aligned}
$$
where $\CM_0$ is defined as follows
\begin{equation}\begin{aligned}
\CM_0 &=  \max_{\gamma \in \{1,\frac{s-k}{s-2k}\}} \Biggl\{  \left( 2^{\frac{1}{2}k^2+\frac{31}{6}k+7}e^{\frac{3}{4}k^2-\frac{1}{2}k}k^{-\frac{1}{2}k^2+\frac{25}{3}k} \right)^{\gamma}, 2^{-\frac{1}{2}k^2-\frac{11}{6}k} \Biggr \} \\
&= \begin{cases} \left( 2^{\frac{1}{2}k^2+\frac{31}{6}k+7}e^{\frac{3}{4}k^2-\frac{1}{2}k}k^{-\frac{1}{2}k^2+\frac{25}{3}k} \right)^{\frac{s-k}{s-2k}}, & k\le 43, \\ 2^{\frac{1}{2}k^2+\frac{31}{6}k+7}e^{\frac{3}{4}k^2-\frac{1}{2}k}k^{-\frac{1}{2}k^2+\frac{25}{3}k}, & 44\le k \le 62, \\
2^{-\frac{1}{2}k^2-\frac{11}{6}k}, & k \ge 63. \end{cases}
\label{eq:M0}
\end{aligned}\end{equation}
\label{thm:final}
\end{thm}
Iterating this theorem combined with the Hardy--Littlewood method one may conclude the following result, which is a special case of Theorem \ref{thm:ultimo}.
\begin{thm} Let $k \ge 3$, $s \ge \frac{5}{2}k^2+k$. Furthermore let $X \ge s^{10}$. Then we have the estimate
$$
J_{s,k}(X) \le CX^{2s-\frac{1}{2}k(k+1)},
$$
where $C$ is the maximum of $4k^{30k^3}$ and
$$\begin{aligned}
& \Biggl[ 2^{\frac{3}{2}k^2+\frac{11}{2}k+1+D}k^{\frac{1}{2}k^2+\frac{25}{6}k-2+D} \CM_0 \Biggr]^{\frac{33}{10} k^{D+1}} \cdot 4(2k)^{2k^3+11k^2},
\end{aligned}$$
where 
$$
D= \left \lceil \frac{2 \log(k)+\log(\log(k))+4.2}{\log(2)} \right \rceil 
$$
and $\CM_0$ as in \eqref{eq:M0}.
\label{thm:intro}
\end{thm}

Although Wooley's method is in principle effective and can be made effective in a similar fashion as we do here, it has a rather big disadvantage. Namely the conditioning and the congruencing step get into each others way. This may be seen best in \cite[Section 7]{WooleyAnnals}, where the sequence $\{b_n\}_n$ follows the iteration scheme $b_{n+1}=kb_n+h_n$ and a lot of effort is put into showing that this sequence doesn't grow too fast. A consequence of this is that the parameter $\theta$ has to be smaller by a factor $2$, which may be further improved down to $\frac{4}{3}$. However this decrease in $\theta$ affects the speed of convergence of $\eta_{s,k}$ drastically (see Theorem \ref{thm:final}). Using the techniques of Vinogradov and Hua instead gains us an independence of the conditioning/well-spacing and the congruencing/boxing step, which leaves us with a simple iteration scheme $b_{n+1}=kb_n$. It is this simple iteration scheme which makes the rather basic outline of the proof in Section \ref{sec:outline} clean. Clean in the sense that simply specifying the involved parameters as well as analysing the dependence in $\vect{g}$, which simply boils down to an exponent being non-positive
, would yield a complete proof. This is not the case when working with congruences. Another novelty of the simple iteration scheme is that it allows the introduction of the parameter $\lambda = \frac{s-k}{k^2}$. This parameter has a large impact on the number of iterations needed in order to decrease the exponent, which has a welcoming effect on the constant. From Theorem \ref{thm:final} it can easily be seen that choosing $\lambda > 1$ rather than $\lambda=1$ decreases the number of iterations $D$ from polynomially in $k$ down to logarithmically in $k$. Which has the effect of reducing the constant from $k^{k^{O(k^2/\epsilon)}}$ down to $k^{k^{O(\log(k^2/\epsilon)/\log(\lambda))}}$ if one wishes to achieve an exponent $\eta_{s,k} \le \epsilon$. This explains why we chose to present Theorem \ref{thm:intro} with a slightly larger $s$ than the method would allow.

In view of further improvements in efficient congruencing and the recent breakthrough by Bourgain, Demeter and Guth one might ask to which extent they can be made effective and how such a result would compare to Theorem \ref{thm:intro}. Let us first remark that the proof of Bourgain, Demeter and Guth follows a similar iteration scheme as multigrade efficient congruencing. One iteration of theirs shows that $V_{p,n}(\delta)$ can be replaced by
$$
\delta^{-\frac{u}{2}} V_{p,n}(\delta)^{1-uW},
$$
whereas efficient congruencing/boxing shows that $X^{\eta}$ can be replaced by
$$
X^{\eta}(X^{\Delta\theta}+X^{-\eta\theta \frac{s-2k}{s-k}})
$$
(see Proposition \ref{prop:oneit}). Now their parameter $u$ can be compared with the parameter $\theta$ in efficient congruencing as in a sense that they have to look at the tiniest scales as well in order for their iteration scheme to go through. Furthermore the factor $\delta^{-\frac{u}{2}}$ can be seen as the defect of the method, which one also gets in efficient congruencing (compare to the positive term in $\Delta$). Now the tinier $V_{p,n}(\delta)$ gets the larger $W$ has to be and henceforth more iterations are needed to further decrease $V_{p,n}(\delta)$. This is the same kind of problem that also comes with effective versions of efficient congruencing/boxing.

Now we should remark that in most applications of Vinogradov's mean value theorem it is important that $s$ is above the critical case; i.e. $s \ge \frac{1}{2}k(k+1)$. Later versions of efficient congruencing as well as the proof of the main conjecture by Bourgain, Demeter and Guth attack the problem from below which corresponds to the case $\lambda=1$. Therefore constants of the size $k^{k^{O(k^2/\epsilon)}}$ should be expected. But let us suppose now that Bourgain, Demeter and Guth's proof could be adapted to an attack from above. This would then lead to a doubling of the parameter $\lambda$, which would speed up the rate of convergence significantly and thus decrease the implied constant.

It is natural to ask if and to what extent Theorem \ref{thm:intro} leads to improvements of the explicit zero-free region of the Riemann-Zeta function, as in Ford's work \cite{FordAnnals}. Unfortunately the answer is that there are no direct improvements. The reason for this is the growth of the constant. The dominating term is roughly $k^{k^{O(\log(k)/\log(\lambda))}}$, which is a lot bigger compared to the term $k^{O(k^3)}$ appearing in Ford's work. When it comes to its application one only takes the $k^{4}$-th root and thus the constant is too large. There may however be a way around this. Similar to the argument in \cite[Section 2]{WooleyAnnals} one may choose $D=1$ and replace Proposition \ref{prop:endI} with one that bounds the quantity at hand in terms of $J_{s-k}(X^{1-\frac{1}{k}})$. This leads to an error of the exponent morally of the size $\eta_{s,k}=k^2 e^{-\frac{1}{2}(s/k^2)^2}$, whilst keeping the constant on the scale $k^{O(k^4)}$. Then taking the $s^2$-th root, with $s=k^2\log(k)^{\frac{1}{2}}$ rather than $s=k^2$, leads to a zero-free region, which is asymptotically slightly worse than Ford's explicit zero-free region. It is therefore not clear if such an endeavour would be fruitful and lead to an improved zero-free region of the Riemann-Zeta function in an intermediate range.

\begin{acknowledgements} I would like to thank my supervisors Andrew Booker and Tim Browning for introducing me to the problem, Trevor Wooley and Kevin Ford for helpful discussions on the topic and one more big thank you to Tim Browning for detailed read-throughs of earlier versions as well as valuable comments along the way. I further wish to thank the referee for their thorough read-through as I know that this paper is not easy to read and their work is very much appreciated.
\end{acknowledgements}

\section{Notation}

As already introduced we let
$$
f(N,M,\vect{\alpha})=\sum_{N-\frac{1}{2}M < x \le N+\frac{1}{2}M} e(\vect{\alpha} \cdot \vect{\vartheta}(x))
$$
for real $N,M$ with $M\ge 1$, where $\vect{\vartheta}(x)=(x,x^2,\dots,x^k)$. Furthermore we call an interval of the shape $]\vect{a},\vect{b}]=]a_1,b_1] \times \dots \times ]a_n,b_n]$ a box. By $\FB^n(\vect{N},\vect{M})$ we denote the box
$$
\FB^n(\vect{N},\vect{M})=\prod_{i=1}^n \left]N_i-\frac{1}{2}M_i,N_i+\frac{1}{2}M_i\right],
$$
furthermore we allow ourselves to abuse some notation here: Any numbers, say $N,M$ in the argument of $\FB^n(N,M)$ are to be regarded as $n$-dimensional vectors with entry $N$, respectively $M$, in each coordinate. To a box $\FB^n(\vect{N},\vect{M})$ we associate the product
\begin{equation}
\FF^n(\vect{N},\vect{M},\vect{\alpha})= \prod_{i=1}^n f(N_i,M_i,\vect{\alpha}).
\label{eq:boxprod}
\end{equation}
We say a box $\FB^n(\vect{N},\vect{M})$ is $R$-well-spaced, if $|N_{i}-N_{j}|\ge 2R$ for all $i\neq j$ and $1 \le M_i\le R$ for $i=1,\dots,n$. In this case we adjust the definition \eqref{eq:boxprod} to
$$
\FF_R^n(\vect{N},\vect{M},\vect{\alpha})= \prod_{i=1}^n f(N_i,M_i,\vect{\alpha})
$$
to indicate further that the box $\FB^n(\vect{N},\vect{M})$ is $R$-well-spaced. We say a box $\FB^n(\vect{N},\vect{M})$ contains a ($k$-dimensional) $R$-well-spaced box if there is a set of $k$ integers $1 \le l_1 <\dots <l_k\le n$ such that $\prod_{i=1}^k \FB^1(N_{l_i},M_{l_i})$ is an $R$-well-spaced box. For such a box we are able to split up the product \eqref{eq:boxprod} into
$$
\FF^n(\vect{N},\vect{M},\vect{\alpha})=\FF_R^k(\vect{N'},\vect{M'},\vect{\alpha})\FF^{n-k}(\vect{N''},\vect{M''},\vect{\alpha}).
$$
The choice of $\vect{N'},\vect{M'},\vect{N''},\vect{M''}$ may of course not be unique.

Other than the initial diophantine equation \eqref{eq:vinomean} we need to consider two more related systems of equations. The first system of equations is

\begin{equation*}
\sum_{i=1}^k (x_i^j-y_i^j) + \sum_{i=1}^{s-k} (u_i^j-v_i^j) = 0, \quad (j=1,\dots,k),
\end{equation*}
where $\vect{x},\vect{y}$ are tuples inside an $R$-well-spaced box $\FB^k(\vect{N},M)$ with $M \ge 1$, $\vect{u},\vect{v}$ are tuples inside a box $\FB^{s-k}(\xi,P)$ for some $\xi \in [-\frac{1}{2},\frac{1}{2}]$ and $P \ge 1$, and furthermore $\FB^k(\vect{N},M) \times \FB^{s-k}(\xi,P) \subseteq ]Q,Q+X]^s$ for some $Q$. Note this forces $-X \le Q \le 0$ as $\vect{0} \in \FB^{s-k}(\xi,P)$. The corresponding counting integral is
$$
I_R(\vect{N},M,\xi,P) = \int_{[0,1[^k}|\FF_R^k(\vect{N},M,\vect{\alpha})|^2 |f(\xi,P,\vect{\alpha})|^{2(s-k)} d\vect{\alpha}.
$$
Let $I_R(M,P)$ denote the maximal number of solutions to the system of equations that occurs for any admissible $\xi$ and $\vect{N}$ given $R,M,P$. Note that $I_R(M,P)$ is certainly bounded by $J_{s,k}(X+1)$ (by considering any solution $(\vect{x},\vect{u}),(\vect{y},\vect{v})\in ]Q,Q+X]^s$ and using Lemma \ref{lem:inttransinv}) and is an integer, therefore well-defined.

The second supplementary system of equations is
\begin{equation*}
\sum_{i=1}^k (x_i^j-y_i^j) + \sum_{i=1}^k (w_i^j-z_i^j) + \sum_{i=1}^{m-k}(u_i^j-v_i^j) + \sum_{i=1}^{s-m-k} (p_i^j-q_i^j) =0 , \quad (j=1,\dots,k),
\end{equation*}
where $\vect{x},\vect{y}$ are tuples inside an $R$-well-spaced box $\FB^k(\vect{N},M)$ with $M \ge 1$, $\vect{w},\vect{z}$ are tuples inside an $R'$-well-spaced box $\FB^k(\vect{N'},L)$ with $L \ge 1$, $\vect{u},\vect{v}$ are tuples inside a box $\FB^{m-k}(N'',L)$, $\vect{p},\vect{q}$ are tuples inside a box $\FB^{s-m-k}(\xi,P)$ for some $\xi \in [-\frac{1}{2},\frac{1}{2}]$ and $P \ge 1$, and furthermore $\FB^k(\vect{N'},L) \subseteq \FB^k(\xi,P)$, $\FB^{m-k}(N'',L) \subseteq \FB^{m-k}(\xi,P)$ and $\FB^k(\vect{N},M) \times \FB^{s-k}(\xi,P) \subseteq ]Q,Q+X]^s$ for some $Q$. The corresponding counting integral is
$$\begin{aligned}
&K_{R,R';m}(\vect{N},M,\vect{N'},L,N'',\xi,P) \\ & \quad = \int_{[0,1[^k}| \FF_R^k(\vect{N},M,\vect{\alpha})|^2 |\FF_{R'}^k(\vect{N'},L,\vect{\alpha})|^2 |f(N'',L,\vect{\alpha})|^{2(m-k)} |f(\xi,P,\vect{\alpha})|^{2(s-m-k)} d\vect{\alpha}.
\end{aligned}$$
Let $K_{R,R';m}(M,P,L)$ denote the maximal number of solutions to the system of equations that occurs for any admissible $\xi,\vect{N},\vect{N'},N''$ given $R,R',m,M,P,L$. Again, this is well-defined.

As only very special types of these two integrals appear we will shorten our notation to
$$\begin{aligned}
I_{a,b}^g(X) &= I_{2^{-g}X^{1-a\theta}}(2^{-g}X^{1-a\theta},X^{1-b\theta}),\\
K_{a,b;m}^{g,h}(X) &= K_{2^{-g}X^{1-a\theta},2^{-h}X^{1-b\theta};m}(2^{-g}X^{1-a\theta},X^{1-b\theta},2^{-h}X^{1-b\theta}),
\end{aligned}$$
where $\theta$ is a sufficiently small parameter, taking on the role already mentioned in the introduction. The parameters $g$ and $h$ indicate the well-spacedness of our boxes and are very important for the argument given in the introduction. \\

The process of getting better and better upper bounds is of an iterative nature, where in each step we decrease the exponent by a tiny bit. In every iteration we will make use of previous upper bounds, thus we assume we have a bound of the shape
\begin{equation}
J_{s,k}(X) \le C \log_2(2X)^{\delta} X^{2s-\frac{1}{2}k(k+1)+\eta},
\label{eq:prevup}
\end{equation}
with $0 \le \delta$ and $0<\eta \le \frac{1}{2}k(k+1)$. Here $\log_2$ denotes the logarithm to the base $2$. Note that we certainly have such a bound with $(C,\delta,\eta)=(1,0,\frac{1}{2}k(k+1))$.

To simplify calculations we introduce the following normalisations:
\begin{equation}\begin{aligned}
\llbracket J_{s,k}(X) \rrbracket &= \frac{J_{s,k}(X)}{C \log_2(2X)^{\delta} X^{2s-\frac{1}{2}k(k+1)+\eta}},\\
\llbracket I_{a,b}^{g}(X) \rrbracket &= \frac{I_{a,b}^g(X)}{C \log_2(2X)^{\delta} (X^{1-a\theta})^{2k-\frac{1}{2}k(k+1)}(X^{1-b\theta})^{2(s-k)}X^{\eta}},\\
\llbracket K_{a,b;m}^{g,h}(X) \rrbracket &= \frac{K_{a,b;m}^{g,h}(X)}{C \log_2(2X)^{\delta} (X^{1-a\theta})^{2k-\frac{1}{2}k(k+1)}(X^{1-b\theta})^{2(s-k)}X^{\eta}}. 
\end{aligned}
\label{eq:normalisation}
\end{equation}
Our assumed upper bound \eqref{eq:prevup} is now reduced to the inequality $\llbracket J_{s,k}(X) \rrbracket  \le 1$, which we will make use of rather frequently. To further simplify our proof we adopt the rather unusual convention that
$$
\displaystyle \sum^W
$$
denotes a sum of at most $W$ terms. In each term the variables
$$
\vect{N},\vect{N'},\vect{N''},\vect{N'''},\vect{N''''},N,N',N_i,N_i', \vect{U}, \vect{U'}, \vect{V}, \xi
$$
may vary, though they are still required to satisfy certain properties coming from the context. These properties include but are not limited to ones such as `being $R$-well-spaced' and `being contained in a box of the shape $]Q,Q+X]^s$' and should always be clear from the context.

\section{Outline of the Proof}
\label{sec:outline}
To give the reader a better understanding of the whole argument we give an overview of what is going on. Recall our assumption \eqref{eq:prevup} and our normalisation \eqref{eq:normalisation}. We have
$$
\llbracket J_{s,k}(X) \rrbracket \le 1
$$
and if $\eta>0$ we would like to show
$$
\llbracket J_{s,k}(X) \rrbracket \ll_{s,k} X^{-\Delta}
$$
for some $\Delta>0$ as large as possible. In the first step we need to ensure that our variables are well-spaced. Secondly we need to start the extraction, by making some variables small. Proposition \ref{prop:initial} does both of these things and essentially gives
\begin{equation}
\llbracket J_{s,k}(X) \rrbracket \ll_{s,k,g} \log_2(2X) \llbracket I_{0,1}^g(X) \rrbracket.
\label{eq:outline1}
\end{equation}
Before extracting information it is better to pre-well-space some variables for further extraction. This is done by Proposition \ref{prop:conditioning} giving essentially
\begin{equation}
\llbracket I_{a,b}^g(X) \rrbracket \ll_{s,k,g,h,m} \log_2(2X) \llbracket K_{a,b;m}^{g,h}(X) \rrbracket.
\label{eq:outline2}
\end{equation}
Now that everything is prepared we can extract some information whilst gaining something in the exponent. This is done by the argument given in the introduction (see Proposistion \ref{prop:extraction} for details). This gives
\begin{equation}
\llbracket K_{a,b;m}^{g,h}(X) \rrbracket  \ll_{s,k,g,h,m} X^{-\eta \frac{s-2k}{s-k} b\theta} \llbracket I_{b,kb}^h(X) \rrbracket^{\frac{k}{s-k}}.
\label{eq:outline3}
\end{equation}
In the end we don't need to pre-well-space any more as it will be the last extraction. After the extraction we bound the number of solutions trivially in terms of $J_{s,k}$. This is done in Proposition \ref{prop:endI} giving the inequality
\begin{equation}
\llbracket I_{a,b}^g(X) \rrbracket \ll_{s,k,g} X^{-\eta\frac{s+k^2-k}{s}b\theta}X^{\frac{k^2(k^2-1)}{2s}b\theta}.
\label{eq:outline4}
\end{equation}
The idea is now to iterate through \eqref{eq:outline2} and \eqref{eq:outline3} as much as possible having fixed $\theta$. By doing so we see that the $h$ cropping up in \eqref{eq:outline2} will become the new $g$ after \eqref{eq:outline3} in the next iteration of \eqref{eq:outline2}, thus we'll get a sequence $\vect{g}$ on which the implied constants will depend. Moreover we see that the pair $(a,b)$ goes through the sequence $(0,1),(1,k),(k,k^2),\dots,(k^{D-1},k^{D})$. For simplicity let us denote this sequence $(a_0,b_0),\dots,(a_D,b_D)$. It turns out that to go through this many iterations one needs $\theta \le k^{-(D+1)}$. So let us fix $\theta=k^{-(D+1)}$ and write
\begin{equation*}\begin{aligned}
 \llbracket J_{s,k}(X) \rrbracket & = \frac{\llbracket J_{s,k}(X) \rrbracket}{\llbracket I_{0,1}^{g_0}(X) \rrbracket} \prod_{n=0}^{D-1} \left( \frac{ \llbracket I_{a_{n},b_{n}}^{g_n}(X) \rrbracket }{ \llbracket I_{a_{n+1},b_{n+1}}^{g_{n+1}}(X) \rrbracket^{\frac{k}{s-k}} } \right)^{\left( \frac{k}{s-k} \right)^n} \llbracket I_{a_{D},b_D}^{g_D}(X) \rrbracket^{\left( \frac{k}{s-k} \right)^{D}}.
\end{aligned}
\end{equation*}
Inserting the equations \eqref{eq:outline1},\eqref{eq:outline2},\eqref{eq:outline3} and \eqref{eq:outline4} we get
\begin{equation*}\begin{aligned}
\llbracket J_{s,k}(X) \rrbracket \ll_{s,k,\vect{g},\vect{m}}& \log_2(2X) \prod_{n=0}^{D-1} \left( \log_2(2X) X^{-\eta\frac{s-2k}{s-k}k^n\theta} \right)^{\left( \frac{k}{s-k} \right)^n}\\
& \cdot \left( X^{-\eta\frac{s+k^2-k}{s}k^D\theta}X^{\frac{k^2(k^2-1)}{2s}k^D\theta} \right)^{\left( \frac{k}{s-k} \right)^D}\\
\ll_{s,k,\vect{g},\vect{m}}& \log_2(2X)^{\frac{2s-3k}{s-2k}} \left(X^{\theta}\right)^{\frac{k^2(k^2-1)}{2s}(\frac{k^2}{s-k})^D-\eta\frac{s-2k}{s-k} \sum_{n=0}^D \left( \frac{k^2}{s-k} \right)^n },
\end{aligned}
\end{equation*}
where we have made use of the trivial inequality $\frac{s+k^2-k}{s}\ge \frac{s-2k}{s-k}$ to make things simpler. Furthermore we have extended the product to infinity to bound the exponent of the logarithm. It is now evident that if $s \ge k^2+k$ and $\eta>0$ we are able to find a sufficiently large $D$, such that the exponent is negative. This is of course provided we can find a suitable choice of $\vect{g}$ and $\vect{m}$.

\section{Preliminaries}
In this section we collect all lemmata which are needed to prove the core propositions in the next section. The first lemma is essential in almost every step and we will refer to it as the integer translation invariance.

\begin{lem}[Integer translation invariance] For $l \in \BZ$ we have
$$\begin{aligned}
\int_{[0,1[^k} \FF^s(\vect{N},\vect{M},\vect{\alpha}) \FF^s(\vect{N'},\vect{M'},-\vect{\alpha}) d\vect{\alpha} = \int_{[0,1[^k} \FF^s(\vect{N}+l,\vect{M},\vect{\alpha}) \FF^s(\vect{N'}+l,\vect{M'},-\vect{\alpha}) d\vect{\alpha}.
\end{aligned}$$
\label{lem:inttransinv}
\end{lem}
\begin{proof}
The first integral is counting the number of integer solutions to
\begin{equation}
\sum_{i=1}^s x_i^j = \sum_{i=1}^s y_i^j \quad (j=1,\dots,k),
\label{eq:count1}
\end{equation}
with $\vect{x} \in \FB^s(\vect{N},\vect{M})$ and $\vect{y} \in \FB^s(\vect{N'},\vect{M'})$. By the Binomial Theorem the system of equations \eqref{eq:count1} is equivalent to
\begin{equation*}
\sum_{i=1}^s (x_i+l)^j = \sum_{i=1}^s (y_i+l)^j \quad (j=1,\dots,k),
\end{equation*}
with $\vect{x}+l \in \FB^s(\vect{N}+l,\vect{M})$ and $\vect{y}+l \in \FB^s(\vect{N'}+l,\vect{M'})$. This is exactly the corresponding diophantine equation of the second integral and we have shown that translating by $l$ gives a one to one correspondence between the two, hence the number of solutions are equal. \end{proof}

\begin{lem} For $x \ge 1$ we have
$$
\sqrt{2 \pi} x^{x+\frac{1}{2}}e^{-x} \le \Gamma(x+1) \le e x^{x+\frac{1}{2}}e^{-x}.
$$
\label{lem:gamma}
\end{lem}

\begin{proof} Despite there being a vast literature on inequalities involving the Gamma-function the author was unable to find a reference for the above inequality, hence we provide a proof. Consider the function $f(x)=\log( \Gamma(x+1) ) - (x+\frac{1}{2})\log(x)+x$. From \cite{Gammaineq} we know that
$$
f''(x)=\sum_{k=1}^{\infty} \frac{1}{(k+x)^2}-\frac{1}{x}+\frac{1}{2x^2} > \frac{1}{6x^3}-\frac{1}{30x^5}>0, \quad (x\ge 1).
$$
Thus $f(x)$ is convex for $x \ge 1$. Moreover we have $f'(1)=\frac{1}{2}-\gamma<0$, where $\gamma$ is the Euler-Mascheroni constant, and
$$
\lim_{x \to \infty}f(x)=\frac{1}{2}\log(2\pi)
$$
from Stirling's approximation. Since $f(x)$ is convex it follows that
$$
\lim_{x \to \infty} f'(x) = 0.
$$
Again from the convexity it follows that $f'(x) < 0$ for $x \ge 1$. Hence the maximum is attained at $x=1$ and the minimum at infinity. This gives the desired inequality.\end{proof}

\begin{lem}
Let $S>0$ be a real number. Further let $1 \le u \le u+r-1$ and let $\FB^r(\vect{N},M)$ be an $R$-well-spaced box with $N_1\le N_2 \le \dots \le N_r$. Suppose we are given two real $r$-tuples $\vect{x},\vect{y} \in \FB^r(\vect{N},M)$ with $-X<\vect{x},\vect{y}\le X$ such that
$$
\left|\sum_{i=1}^r x_i^j - \sum_{i=1}^r y_i^j\right| \le SX^{j-1}
$$
holds for every $j=u,\dots,u+r-1$. If $u>1$ we furthermore need the assumption $x_jy_j>0$ for $j=1,\dots,r-1$ as well as $|x_r|,|y_r| \ge U > 0$. Then we have
$$
|x_r-y_r| \le \sqrt{2} \left( \frac{e}{r} \right )^r \left( \frac{X}{R} \right)^{r-1} \left(  \frac{X}{U}\right)^{u-1}S.
$$
\label{lem:smallint}
\end{lem}
\begin{proof} We follow Hua's argument quite closely (See \cite{AToPN} Lemma 1 page 181-183). We write
$$
\sum_{i=1}^r \frac{x_i^j-y_i^j}{x_i-y_i} \cdot (x_i-y_i) = \theta_j X^{j-1}, \quad u\le j \le u+r-1,
$$
where $|\theta_j| \le S$ for every $u\le j \le u+r-1$. We regard this as linear system of equations in $x_1-y_1,\dots, x_r-y_r$. By Cramer's rule we have
\begin{equation}
\Delta (x_r-y_r) - \Delta'=0,
\label{eq:cramer}
\end{equation}
where
$$\begin{aligned}
\Delta &= \begin{vmatrix} \frac{x_1^u-y_1^u}{u(x_1-y_1)} & \dots & \frac{x_r^u-y_r^u}{u(x_r-y_r)} \\ \vdots & & \vdots \\  \frac{x_1^{u+r-1}-y_1^{u+r-1}}{(u+r-1)(x_1-y_1)} & \dots & \frac{x_r^{u+r-1}-y_r^{u+r-1}}{(u+r-1)(x_r-y_r)} \end{vmatrix}, \\
\Delta'&= \begin{vmatrix} \frac{x_1^u-y_1^u}{u(x_1-y_1)} & \dots & \frac{x_{r-1}^u-y_{r-1}^u}{u(x_{r-1}-y_{r-1})} & \frac{\theta_{u}}{u}X^{u-1} \\ \vdots & & \vdots & \vdots \\  \frac{x_1^{u+r-1}-y_1^{u+r-1}}{(u+r-1)(x_1-y_1)} & \dots & \frac{x_{r-1}^{u+r-1}-y_{r-1}^{u+r-1}}{(u+r-1)(x_{r-1}-y_{r-1})} & \frac{\theta_{u+r-1}}{u+r-1} X^{u+r-2} \end{vmatrix}.
\end{aligned}$$
Now one can rewrite \eqref{eq:cramer} as
$$
\frac{1}{\prod_{i=1}^r(x_i-y_i)} \int_{y_1}^{x_1}\dots\int_{y_r}^{x_r} \left( \Delta_{u,r}(x_r-y_r) - \Delta'_{u,r} \right) dz_1\dots dz_r =0,
$$
where
$$\begin{aligned}
\Delta_{u,r} &= \begin{vmatrix} z_1^{u-1} & \dots & z_r^{u-1}  \\ \vdots & & \vdots \\ z_1^{u+r-2} & \dots & z_r^{u+r-2} \end{vmatrix}, \\
\Delta'_{u,r} &= \begin{vmatrix} z_1^{u-1} & \dots & z_{r-1}^{u-1} & \frac{\theta_u}{u}X^{u-1}  \\ \vdots & & \vdots & \vdots \\ z_1^{u+r-2} & \dots & z_{r-1}^{u+r-2} & \frac{\theta_{u+r-1}}{u+r-1} X^{u+r-2} \end{vmatrix}.
\end{aligned}$$
In the case of $x_i=y_i$ for some $i$ we can still make sense of the above argument in terms of limits, which do exist. By the mean-value theorem of integral calculus there is a choice of $z_i \in [x_i,y_i]$ for $1\le i \le r$ such that
$$
\Delta_{u,r}(x_r-y_r) - \Delta'_{u,r}=0.
$$
By considering Vandermonde determinants we find the identity
$$
\Delta_{u,r}= \Delta_{u,r-1} \cdot z_r^{u-1} \prod_{i=1}^{r-1}(z_r-z_i)
$$
and moreover
$$
\Delta_{u,r-1} \neq 0
$$
as the $z_i$ are pairwise different and in the case $u>1$ we also have $z_i \neq 0$ as $0 \nin [x_i,y_i]$ for $i=1,\dots,r-1$. Let us denote the elementary symmetric polynomial of degree $r-i$ in the variables $z_1,\dots,z_{r-1}$ by $\sigma_{r-i}$. These satisfy $|\sigma_{r-i}| \le \binom{r-1}{r-i}X^{r-i}$ as $-X \le \vect{z} \le X$. Therefore, in the expansion of $\Delta_{u,r}$, the absolute values of the coefficient of $z_r^{u+i-2}$ are equal to
$$
|\sigma_{r-i}\Delta_{u,r-1}| \le \binom{r-1}{r-i}X^{r-i}|\Delta_{u,r-1}|.
$$
Using the column minor of expansion of $\Delta'_{u,r}$ and comparing it with the corresponding one of $\Delta_{u,r}$ we find that
$$\begin{aligned}
|\Delta'_{u,r}| &\le |\Delta_{u,r-1}| \sum_{i=1}^{r}\frac{|\sigma_{r-i}||\theta_{u+i-1}|}{u+i-1} X^{u+i-2} \le |\Delta_{u,r-1}| S X^{u+r-2} \sum_{i=1}^{r} \frac{1}{u+i-1} \binom{r-1}{r-i} \\
& \le \frac{2^r}{r} \cdot |\Delta_{u,r-1}| \cdot SX^{u+r-2},
\end{aligned}$$
since
$$
\sum_{i=1}^{r}\frac{1}{u+i-1}\binom{r-1}{r-i} \le \sum_{i=1}^{r} \frac{1}{i}\binom{r-1}{r-i} = \sum_{i=1}^{r} \frac{1}{r}\binom{r}{r-i} \le \frac{2^{r}}{r}.
$$
It follows that
$$
|x_r-y_r| \le \frac{2^r \cdot  SX^{u+r-2}}{r \cdot |z_r|^{u-1} \prod_{i=1}^{r-1} (z_r-z_i)} \le \frac{2^r}{r \cdot \prod_{i=1}^{r-1}(2i-1)} \left( \frac{X}{R} \right)^{r-1} \left( \frac{X}{U} \right)^{u-1}S,
$$
where we have used $|z_r-z_{r-i}|\ge (2i-1)R$ and $|z_r| \ge U$. Furthermore we have for $r\ge 3$
$$\begin{aligned}
\prod_{i=1}^{r-1}(2i-1)&= \frac{\Gamma(2r-1)}{\Gamma(r) \cdot 2^{r-1}} = \frac{2^{r-1}}{\sqrt{\pi}} \Gamma\left(r-\frac{1}{2}\right) \\
& \ge 2^{r-\frac{1}{2}} \left(r-\frac{3}{2}\right)^{r-1} e^{\frac{3}{2}-r} \\
& \ge 2^{r-\frac{1}{2}} r^{r-1} e^{-r},
\end{aligned}$$
where we have made use of Lemma \ref{lem:gamma}. It is easily checked, that this inequality also holds for $r=2$. Thus we have
$$
|x_r-y_r| \le \sqrt{2} \left( \frac{e}{r} \right)^r \left(\frac{X}{R}\right)^{r-1} \left( \frac{X}{U} \right)^{u-1}S,
$$
which holds also for $r=1$ for trivial reasons.\end{proof}

\begin{lem} We have for $r \ge 1$:
$$
\prod_{n=1}^r n^{n-1}\ge r^{\frac{1}{2}r(r-2)} e^{-\frac{1}{4}(r-1)(r-3)}.
$$
\label{lem:prod}
\end{lem}
\begin{proof} The function $(x-1) \log(x)$ is convex with a minimum of $0$ at $1$, thus
$$\begin{aligned}
\prod_{n=1}^r n^{n-1} &= \exp\left(\sum_{n=1}^r (n-1) \log(n)\right) \\
&\ge \exp\left( \int_1^r (x-1)\log(x) dx \right) \\
&= r^{\frac{1}{2}r(r-2)} e^{-\frac{1}{4}(r-1)(r-3)}. 
\end{aligned}$$ \end{proof}

\begin{lem} Let $1 \le r \le k$ and furthermore let $\FB^r(\vect{N},M) \subseteq ] -X,X]^r$ be an $R$-well-spaced box with $N_1\le N_2 \le \dots \le N_r$ and $M, S \ge 1$. Assume as well $R\le X/(2k)$. Let $\CZ_W(\FB^r(\vect{N},M), \vect{U})$ be the number of integer solutions $\vect{x}\in \FB^r(\vect{N},M)$ counted with multiplicity $W(\vect{x})\ge 0$ satisfying
\begin{equation}
\sum_{i=1}^r x_i^j \in U_j \quad (j=1,\dots,r),
\label{eq:inti}
\end{equation}
where $U_j$ is an interval of size at most $SX^{j-1}$. Then we have the bound
$$\begin{aligned}
\CZ_W(\FB^r(\vect{N},M), \vect{U}) \le& 2^{\frac{1}{2}r(r+1)} e^{\frac{1}{4}(3r+1)(r-1)} r^{-\frac{1}{2}r(r-2)} \cdot \left(\frac{X}{R}  \right)^{\frac{1}{2}r(r-1)} \\
& \cdot \CZ_W(\FB^r(\vect{N'},\vect{S'}),\vect{U}),
\end{aligned}$$
for some sub-box $\FB^r(\vect{N'},\vect{S'})$ of $\FB^r(\vect{N},M)$ with $1 \le \vect{S'}\le S$.
\label{lem:smallboxes}
\end{lem}

\begin{rem} If $S \le M$, then we are of course able to choose $\vect{S'} \equiv S$.
\end{rem}

\begin{proof} This will follow from the inequality
\begin{equation}\begin{aligned}
\CZ_W(\FB^r(\vect{N},M), \vect{U}) \le& 2^{\frac{1}{2}r(r+1)} e^{\frac{1}{2}r(r+1)-1} \prod_{n=1}^{r} n^{-(n-1)} \cdot \left(\frac{X}{R}  \right)^{\frac{1}{2}r(r-1)} \\
& \cdot \CZ_W(\FB^r(\vect{N'},\vect{S'}),\vect{U}),
\end{aligned}\label{eq:indhyp}
\end{equation}
which we shall prove inductively, and Lemma \ref{lem:prod}. \eqref{eq:indhyp} holds clearly for $r=1$, thus we may assume $r \ge 2$ from now on. Without loss of generality we may assume $W(\vect{x})=0$ if $\vect{x}$ does not satisfy \eqref{eq:inti}. If there are no solutions, then \eqref{eq:indhyp} holds trivially, thus assume now there is a solution $\vect{x}$. If we have another solution $\vect{x'}$, we deduce
$$
\left|\sum_{i=1}^{r} x_i^j- \sum_{i=1}^r x_i'^j \right| \le SX^{j-1}
$$
from \eqref{eq:inti}. We can apply Lemma \ref{lem:smallint} with $u=1$ to deduce that
$$
|x_r-x_r'|\le \sqrt{2} \left( \frac{e}{r} \right)^r \left( \frac{X}{R} \right)^{r-1}S.
$$
Thus we can find a sub-box of $\FB^1(N_r,M)$ of size at most
$$
\sqrt{2} \left( \frac{e}{r} \right)^r \left( \frac{X}{R} \right)^{r-1}S + 1 \le 2 \left(\frac{e}{r} \right)^r \left( \frac{X}{R} \right)^{r-1}S
$$
in which all the $x_r$'s lie, as
$$
\left(\frac{X}{R}\right)^{r-1}\left( \frac{e}{r} \right)^r \ge \left(2k\right)^{r-1} \left( \frac{e}{r} \right)^r \ge \frac{1}{2r} (2e)^r \ge e 
$$
and $\sqrt{2}+e^{-1}\le 2$. This box we may further split up into at most
$$
2 \cdot  (r-1) \left( \frac{e}{r} \right)^r \left( \frac{X}{R} \right)^{r-1} + 1 \le 2 r \left( \frac{e}{r} \right)^r \left( \frac{X}{R} \right)^{r-1}
$$
boxes of size at most $\frac{S}{r-1}$. We now consider the interval $S_r'$ say, which contributes the most towards $\CZ_W(\FB^r(\vect{N},M), \vect{U})$. By assumption we still have at least a solution $\vect{x}$. Consider now a second solution $\vect{x'}$. Call $\vect{y}$ and $\vect{y'}$ their restrictions to the first $k-1$ coordinates. Note that we have
$$\begin{aligned}
\left|\sum_{i=1}^{r-1} y_i^j - \sum_{i=1}^{r-1} y_i'^j \right | &\le \left|\sum_{i=1}^{r} x_i^j - \sum_{i=1}^{r} x_i'^j \right|+\left| x_r^j-x_r'^j \right| \\
&\le SX^{j-1}+\frac{S}{r-1}\cdot jX^{j-1} \\
& \le 2SX^{j-1}
\end{aligned}$$
for $j=1,\dots,r-1$. Thus their $j$-th power sum is contained in some interval of length at most $2SX^{j-1}$. Each of these intervals we split into half, yielding
$$
\CZ_W(\FB^r(\vect{N},M), \vect{U}) \le 2 r \left( \frac{e}{r} \right)^r \left( \frac{X}{R} \right)^{r-1} \cdot \sum^{2^{r-1}} \CZ_{W'}(\FB^{r-1}(\vect{N},M), \vect{U'}),
$$
where $W'(\vect{y})=\sum_{x_r \in S_r'} W((\vect{y},x_r))$. We now take the maximum and apply the induction hypothesis for $r-1$ with $W'$. The induction is now complete as
$$\begin{aligned}
2 r \left( \frac{e}{r} \right)^r  \left( \frac{X}{R} \right)^{r-1} \cdot  2^{r-1} \cdot & 2^{\frac{1}{2}r(r-1)} e^{\frac{1}{2}r(r-1)-1} \prod_{n=1}^{r-1} n^{-(n-1)} \cdot \left(\frac{X}{R}  \right)^{\frac{1}{2}(r-1)(r-2)} \\
= &  2^{\frac{1}{2}r(r+1)} e^{\frac{1}{2}r(r+1)-1} \prod_{n=1}^{r} n^{-(n-1)} \cdot \left(\frac{X}{R}  \right)^{\frac{1}{2}r(r-1)}.
\end{aligned}$$
and
$$
\CZ_{W'}(\FB^{r-1}(\vect{N'},\vect{S'}),\vect{U'}) \le \CZ_W(\FB^r(\vect{N'},\vect{S'}),\vect{U}),
$$
where $\FB^r(\vect{N'},\vect{S'})=\FB^{r-1}(\vect{N'},\vect{S'}) \times S_r'$.
\end{proof}

\begin{lem} Let $k,m,D \in \BN$ with $m\ge k \ge 2$. The set of integers $(d_1,\dots,d_m)$ with $0 \le d_i < D$ for $i=1,\dots,m$ is said to contain a well-spaced ($k$-dimensional) subtuple if there are $k$ of them, say $d_{i_1},\dots,d_{i_k}$, satisfying
$$
d_{i_{j+1}}-d_{i_j}>1 \quad j=1,\dots,k-1.
$$
The number of tuples not containing a well-spaced subtuple is bounded by
$$
2^{m}k^{m}D^{k-1}.
$$
\label{lem:notwsp}
\end{lem}


\begin{proof} See \cite[Lemma 4.3]{AToPN}. There is however a slight error in the proof. Their argument gives the bound
$$\begin{aligned}
\sum_{u=1}^{k-1} \binom{D}{u}(2u)^m &\le 2^m (k-1)^m D^{k-1} \sum_{u=1}^{k-1} \frac{1}{u!} \\
&\le 2^m k^m e^{-\frac{m}{k}} D^{k-1} (e-1) \\
& \le 2^m k^m D^{k-1}.
\end{aligned}$$
If one works a little bit harder one may also recover the bound claimed in \cite{AToPN}.
\end{proof}

This next lemma is the key to the well-spacing Propositions \ref{prop:initial} and \ref{prop:conditioning}. It is essentially about bounding the number of solutions in a box $\FB^m(N,P)\times \FB^m(N,P)$ by the number of solutions in its sub-boxes, most of which contain a ($k$-dimensional) $R$-well-spaced box with $R$ being of a large size compared to the sub-box itself.

\begin{lem} For $G \ge 1$, $m \ge k+1, k \ge 2$ and $P\ge 2^G$ we have $|f(N,P,\vect{\alpha})|^{2m}$ is bounded by
$$\begin{aligned}
G \Biggl[& \sum_{g= \lceil \log_2(2k) \rceil}^G \frac{2^mL_{g-1}}{m-k} \sum^{(m-k)2^mL_{g-1}} |\FF_{2^{-g}P}^k(\vect{N'},2^{-g}P,\vect{\alpha})|^2|f(N',2^{-g}P,\vect{\alpha})|^{2(m-k)} \\
&+\frac{L_G}{m} \sum^{mL_G} |f(N',2^{-G}P,\vect{\alpha})|^{2m} \Biggr],
\end{aligned}$$
where all the boxes on the right hand side are contained in the box $\FB^{2m}(N,P)$ and
$$\begin{aligned}
L_g &= 2^{m} k^{m} \cdot (2^g)^{k-1}.
\end{aligned}$$
\label{lem:wellspacing}
\end{lem}

\begin{proof} At the heart of the argument lies the equality
\begin{equation}
f(N,P,\vect{\alpha})=f(N-P/4,P/2,\vect{\alpha})+f(N+P/4,P/2,\vect{\alpha}).
\label{eq:split1}
\end{equation}
Iterating this equality shows that we have
\begin{equation*}
f(N,P,\vect{\alpha})=\sum_{d=0}^{2^g-1} f(N-P/2+(d+1/2)2^{-g}P,2^{-g}P,\vect{\alpha})
\end{equation*}
for every $g \in \BN$. This further leads to
\begin{equation}
f(N,P,\vect{\alpha})^m= \sum_{0 \le \vect{d} \le 2^{g}-1} \prod_{i=1}^m f(N-P/2+(d_i+1/2)2^{-g}P,2^{-g}P,\vect{\alpha}).
\label{eq:prodsplit1}
\end{equation}
Now the above box $\FB^m(N-P/2+(\vect{d}+1/2)2^{-g}P,2^{-g}P)$ contains a ($k$-dimensional) $2^{-g}P$-well-spaced box if and only if the tuple $\vect{d}=(d_1,\dots,d_m)$ contains a ($k$-dimensional) well-spaced subtuple.\\

Our plan is to extract the tuples which contain a well-spaced subtuple from \eqref{eq:prodsplit1} before using \eqref{eq:split1} on the remaining summands. For this purpose we are going to use the binary expansion of the $d$'s. Given $\vect{d} \in \BN_0^m$ we define the predecessor of $\vect{d}$ to be $\vect{p}(\vect{d})=(\lfloor d_1/2 \rfloor,\dots,\lfloor d_m/2 \rfloor)$. We also define the set of successors of $\vect{d}$ as $S(\vect{d})=\{\vect{d'}\in \BN_0^m | \vect{p}(\vect{d'})=\vect{d}\}$. Note that if $\vect{p}(\vect{d})$ contains a ($k$-dimensional) well-spaced subtuple then so does $\vect{d}$. We abbreviate `$\vect{d}$ contains a ($k$-dimensional) well-spaced subtuple' to `$\vect{d}$ is good'. We are now able to prove the following identity for $G \in \BN_0$ inductively:
\begin{equation}\begin{aligned}
f(N,P,\vect{\alpha})^m =& \sum_{g=0}^G \sum_{\substack{0 \le \vect{d} \le 2^g-1\\ \vect{d} \text{ is good}\\ \vect{p}(\vect{d}) \text{ is not good}}} \prod_{i=1}^m f(N-P/2+(d_i+1/2)2^{-g}P,2^{-g}P,\vect{\alpha}) \\
&+ \sum_{\substack{0 \le \vect{d} \le 2^{G}-1 \\ \vect{d} \text{ is not good}}} \prod_{i=1}^m f(N-P/2+(d_i+1/2)2^{-G}P,2^{-G}P,\vect{\alpha}).
\label{eq:wellspac}
\end{aligned}\end{equation}
For $G=0$ the right hand side is just $f(N,P,\vect{\alpha})^m$ as the first sum is empty. The induction step follows from the identity
$$\begin{aligned}
\prod_{i=1}^m  f(N-P/2+&(d_i+1/2)2^{-G}P,2^{-G}P,\vect{\alpha}) \\
& = \sum_{\vect{d'} \in S(\vect{d})} \prod_{i=1}^m f(N-P/2+(d_i'+1/2)2^{-(G+1)}P,2^{-(G+1)}P,\vect{\alpha}),
\end{aligned}$$
which is just \eqref{eq:split1} applied to each factor, applied to the latter sum in \eqref{eq:wellspac}; i.e. the $\vect{d}$'s which are not good, and splitting up into good and not good tuples.

Now we note that if $g \le \lceil \log_2(2k) \rceil-1$ we have that all tuples $0 \le \vect{d}\le 2^{g}-1$ are not good, because if there were a good one we would have
$$
D=2^g \ge 1+d_{i_k}=1+\sum_{j=1}^{k-1} (d_{i_{j+1}}-d_{i_j})+d_{i_1} \ge 1+2(k-1).
$$
If $g \ge 1$ we then have $2^g \ge 2k$ as $2k-1$ is odd, leading to a contradiction. And for $g=0$ there are obviously no good tuples. The next thing we note is the number of not good tuples $0 \le \vect{d} \le 2^{g}-1$ is at most
$$
L_g = 2^{m} k^{m} \cdot (2^g)^{k-1},
$$
by Lemma \ref{lem:notwsp}. Moreover we have $|S(\vect{d})|=2^m$ which shows that the set of tuples $0 \le \vect{d} \le 2^{g}-1$ such that $\vect{d}$ is good and $\vect{p}(\vect{d})$ is not good has cardinality at most $2^mL_{g-1}$. Therefore we conclude that the equality \eqref{eq:wellspac} is of the shape
$$\begin{aligned}
f(N,P,\vect{\alpha})^m =& \sum_{g= \lceil \log_2(2k) \rceil}^G \sum^{2^mL_{g-1}} \FF_{2^{-g}P}^k(\vect{N'},2^{-g}P,\vect{\alpha}) \FF^{m-k}(\vect{N''},2^{-g}P,\vect{\alpha}) \\
&+\sum^{L_G}\FF^{m}(\vect{N'''},2^{-G}P,\vect{\alpha}).
\end{aligned}$$
Applying Cauchy-Schwarz twice yields
$$\begin{aligned}
|f(N,P,\vect{\alpha})|^{2m} \le & G \Biggl[\sum_{g= \lceil \log_2(2k) \rceil}^G \left(\sum^{2^mL_{g-1}} |\FF_{2^{-g}P}^k(\vect{N'},2^{-g}P,\vect{\alpha})| |\FF^{m-k}(\vect{N''},2^{-g}P,\vect{\alpha})|\right)^2 \\
&+\left(\sum^{L_G}|\FF^{m}(\vect{N'''},2^{-G}P,\vect{\alpha})| \right)^2 \Biggr ] \\
\le & G \Biggl[\sum_{g= \lceil \log_2(2k) \rceil}^G 2^mL_{g-1} \! \! \sum^{2^mL_{g-1}} |\FF_{2^{-g}P}^k(\vect{N'},2^{-g}P,\vect{\alpha})|^2 |\FF^{m-k}(\vect{N''},2^{-g}P,\vect{\alpha})|^2 \\
&+L_G\sum^{L_G}|\FF^{m}(\vect{N'''},2^{-G}P,\vect{\alpha})|^2 \Biggr ],
\end{aligned}$$
since $G-\lceil \log_2(2k) \rceil+1+1 \le G$. By further using AM-GM in the shape
$$
|\FF^{r}(\vect{N},M,\vect{\alpha})|^2 \le \frac{1}{r} \sum_{i=1}^r |f(N_i,M,\vect{\alpha})|^{2r},
$$
we prove the desired inequality.\end{proof}

\begin{rem} Potentially one could gain more $\log_2(2X)$ savings if one were to allow mixed terms with different $g$'s, but the state of affairs is already complicated enough as it is and so we omit exploring this possibility.
\end{rem}

\section{Core Propositions}
In this section we prove the core propositions which will be used in the final argument, as outlined in Section \ref{sec:outline}. From now on we will also assume that $k \ge 3$.

This first proposition well-spaces a set of variables in order to get the iteration started.

\begin{prop} Let $G \in \BN$ and assume $X^{\theta} \ge 2^{G}$ with $0<\theta \le \frac{1}{k^2}$. Furthermore let $s \ge m \ge k+1$ and $m \le 8k^2$. Then we have that $\llbracket J_{s,k}(X) \rrbracket$ is bounded by the sum:
$$\begin{aligned}
 C' \cdot G &\sum_{g= \lceil \log_2(2k) \rceil}^{G}  \left(2^{g} \right)^{2(k-1)-2(m-k)}\llbracket I_{0,1}^{g}(X) \rrbracket + C'' \cdot G \cdot \left(2^{G}\right)^{2(k-1)-\frac{m}{s}(2s-\frac{1}{2}k(k+1)+\eta)},
\end{aligned}$$
where
$$\begin{aligned}
C' &= 2^{6m-4k+2}k^{2m} \cdot \left(1+\frac{1}{X^{\theta}} \right)^{2(s-m)},\\
C'' &= 2^{2m+1}k^{2m}.
\end{aligned}$$
\label{prop:initial}\end{prop}

\begin{proof}
We will use Lemma \ref{lem:wellspacing} for $2m$ factors in $J_{s,k}(X)$. We find
$$
J_{s,k}(X) \le G \left[ \sum_{g = \lceil \log_2(2k) \rceil}^G \frac{2^mL_{g-1}}{m-k} \sum^{(m-k)2^mL_{g-1}} \CW_g + \frac{L_G}{m} \sum^{m L_G} \CN \right],
$$
where
$$\begin{aligned}
\CW_g &= \int_{[0,1[^k} |\FF_{2^{-g}X}^k(\vect{N},2^{-g}X,\vect{\alpha})|^2|f(N,2^{-g}X,\vect{\alpha})|^{2(m-k)} |f(X/2,X,\vect{\alpha})|^{2(s-m)} d \vect{\alpha},\\
\CN &= \int_{[0,1[^k} |f(N,2^{-G}X,\vect{\alpha})|^{2m} |f(X/2,X,\vect{\alpha})|^{2(s-m)} d \vect{\alpha}.
\end{aligned}$$
We refer to the first part as the well-spaced part and second part as the non-well-spaced part. Let us first consider the part which is non-well-spaced. There we have to consider the integral $\CN$. We find
\begin{equation*}\begin{aligned}
\CN \le & \left(\int_{[0,1[^k} |f(N,2^{-G}X,\vect{\alpha})|^{2s} d\vect{\alpha} \right)^{\frac{m}{s}} \left( \int_{[0,1[^k}  |f(X/2,X,\vect{\alpha})|^{2s} d\vect{\alpha} \right)^{\frac{s-m}{s}} \\
\le &J_{s,k}(2^{-G}X+1)^{\frac{m}{s}} \cdot J_{s,k}(X)^{\frac{s-m}{s}} 
\end{aligned}\end{equation*}
by H\"older's inequality and the integer translation invariance. Now we have
$$\begin{aligned}
2^{-G}X+1&=2^{-G}X(1+2^GX^{-1}) \le 2^{-G}X \left(1+\frac{9}{2^8k^2}\right)
\end{aligned}$$
as
$$
2^GX^{-1} \le X^{-1+\theta} \le 2^{\theta^{-1}(-1+\theta)} \le 2^{1-k^2} \le \frac{9}{2^8k^2}.
$$
In conclusion the overall contribution to $J_{s,k}(X)$ from the non-well-spaced part is
\begin{equation}
G \cdot 2^{2m}k^{2m}\left(2^G\right)^{2(k-1)}  \cdot J_{s,k}\left(\left(1+\frac{9}{2^8k^2}\right)2^{-G}X\right)^{\frac{m}{s}} \cdot J_{s,k}(X)^{\frac{s-m}{s}},
\label{eq:notwsppart1b}
\end{equation}
where we recalled $L_G=2^mk^m(2^G)^{k-1}$.

For the $2^{-g}X$-well-spaced ones we have to consider the integral $\CW_g$.
We split up the box $\FB^1(N,2^{-g}X)$, respectively $\FB^1(X/2,X)$, into boxes of size at most $X^{1-\theta}$. Moreover we may assume all the boxes have size $X^{1-\theta}$ as it can only happen, that we account for solutions multiple times. There are at most
$$
\frac{2^{-g}X}{X^{1-\theta}}+1 = 2^{-g} X^{\theta} \cdot \left(1+\frac{2^g}{X^\theta} \right) \le 2^{-g}X^{\theta}  \cdot 2 = T_g', \text{ say},
$$
respectively
$$
\frac{X}{X^{1-\theta}}+1 = X^{\theta} \left(1+\frac{1}{X^{\theta}} \right) = T_g'', \text{ say},
$$
boxes of this kind. Thus we have that $\CW_g$ is bounded by
\begin{equation*}\begin{aligned}
& \int_{[0,1[^k} |\FF_{2^{-g}X}(\vect{N},2^{-g}X,\vect{\alpha})|^2 \left|\sum^{T_g'} f(N,X^{1-\theta},\vect{\alpha})\right|^{2(m-k)} \left|\sum^{T_g''} f(N',X^{1-\theta},\vect{\alpha})\right|^{2(s-m)} d\vect{\alpha}, 
\end{aligned}\end{equation*}
which we immediately bound further by using Cauchy-Schwarz as follows
$$\begin{aligned}
&\left|\sum^{T_g'} f(N,X^{1-\theta},\vect{\alpha})\right|^{2(m-k)} \left|\sum^{T_g''} f(N',X^{1-\theta},\vect{\alpha})\right|^{2(s-m)} \\
& \qquad \qquad \le \left(T_g'\sum^{T_g'} |f(N,X^{1-\theta},\vect{\alpha})|^2\right)^{m-k} \left(T_g''\sum^{T_g''} |f(N',X^{1-\theta},\vect{\alpha})|^2\right)^{s-m}.
\end{aligned}$$
If we define
$$
T_g=T_g'^{(m-k)} \cdot T_g''^{(s-m)} = 2^{m-k} \left( 2^{-g} \right)^{m-k} \left( 1+\frac{1}{X^{\theta}} \right)^{s-m} X^{(s-k)\theta}
$$
then we find after expanding the product of sums into a sum of products that
\begin{equation*}
\CW_g \le T_g \int_{[0,1[^k} |\FF_{2^{-g}X}(\vect{N},2^{-g}X,\vect{\alpha})|^2 \left(\sum^{T_g} \prod_{i=1}^{s-k} |f(N_i'',X^{1-\theta},\vect{\alpha})|^2 \right)  d\vect{\alpha}.
\end{equation*}
Using AM-GM and the integer translation invariance we bound the above further by
\begin{equation*}\begin{aligned}
\CW_g \le & \frac{T_{g}}{s-k} \sum^{(s-k)T_{g}} \int_{[0,1[^k} |\FF_{2^{-g}X}(\vect{N},2^{-g}X,\vect{\alpha})|^2 |f(N'',X^{1-\theta},\vect{\alpha})|^{2(s-k)} d\vect{\alpha} \\
=& \frac{T_{g}}{s-k} \sum^{(s-k)T_{g}} \int_{[0,1[^k} |\FF_{2^{-g}X}(\vect{N'},2^{-g}X,\vect{\alpha})|^2 |f(\xi,X^{1-\theta},\vect{\alpha})|^{2(s-k)} d\vect{\alpha}\\
\le& T_{g}^2 \cdot I_{0,1}^{g}(X).
\end{aligned}\end{equation*}
This last step of translating the variables such that some of the variables are small is an essential prerequisite for the extraction argument \eqref{eq:small} to follow. We find that the well-spaced contribution is at most
\begin{equation}
G \sum_{g= \lceil \log_2(2k) \rceil}^{G} 2^{6m-4k+2}k^{2m}\left(2^g \right)^{2(k-1)-2(m-k)} \left( 1+\frac{1}{X^{\theta}} \right)^{2(s-m)} X^{2(s-k)\theta} \cdot I_{0,1}^{g}(X).
\label{eq:wsppart1}
\end{equation}
In conclusion we have that $J_{s,k}(X)$ is bounded by the sum of \eqref{eq:notwsppart1b} and \eqref{eq:wsppart1}.
We now normalise this inequality to get an inequality for $\llbracket J_{s,k}(X) \rrbracket$. After normalising we easily find that the well-spaced part of the proposition is true. In the non-well-spaced part we collect an additional factor of
$$
\frac{\log_2\left(2\left(1+\frac{9}{4k^2} \right)2^{-G}X \right)^{\delta \frac{m}{s}}}{\log_2(2X)^{\delta \frac{m}{s}}} \left(2^{-G}\right)^{\frac{m}{s}(2s-\frac{1}{2}k(k+1)+\eta)} \left(1+\frac{9}{2^8k^2} \right)^{\frac{m}{s}(2s-\frac{1}{2}k(k+1)+\eta)}.
$$
The fraction of $\log$'s is trivially bounded by 1 and since $\eta \le \frac{1}{2}k(k+1)$ and $m \le 8k^2$ we have
$$
\left(1+\frac{9}{2^8k^2} \right)^{\frac{m}{s}(2s-\frac{1}{2}k(k+1)+\eta)} \le \left(1+\frac{9}{2^8k^2} \right)^{2m} \le e^{\frac{9m}{2^7k^2}} \le 2.
$$
This concludes the proof. 
\end{proof}

This next proposition is almost analogous to the previous one. The observant reader may notice an important difference though. Here the well-spacing step takes place two steps ahead of when it is needed. This is advantageous as it allows for a smaller choice of the parameter $m$ in Lemma \ref{lem:wellspacing}.

\begin{prop} Let $H,a,b \in \BN_0$ with $H \ge 1$ and $b>a\ge 0$. Assume $\theta \in \BR$ satisfies $1 \ge k^2b\theta >0$. Let $X \ge 2^{\theta^{-1}}$ and $X^{kb\theta}\ge 2^H$. Let $g \in \BN$ such that $X^{b\theta}\ge 2^g \ge 2k$. Furthermore let $800 k \ge m\ge k+1$ and $6^6k^2 \log(k) \ge s - k\ge m$, then we have that $\llbracket I_{a,b}^{g}(X) \rrbracket$ is bounded by the sum
$$\begin{aligned}
&C' \cdot H \sum_{h=\lceil \log_2(2k)\rceil }^H\left( 2^h \right)^{2(k-1)} \llbracket K_{a,b;m}^{g,h}(X) \rrbracket\\
+&C'' \cdot H \left(2^g\right)^{-\frac{k}{s}(2s-\frac{1}{2}k(k+1)+\eta)} \left( 2^{H} \right)^{2(k-1)-\frac{m}{s}(2s-\frac{1}{2}k(k+1)+\eta)} \\
& \qquad \qquad \qquad \qquad \qquad \qquad \qquad \qquad \cdot \left( X^{\theta} \right)^{\frac{1}{2}k(k+1)\frac{s-k}{s}(b-a)} \left( X^{-\eta \theta} \right)^{\frac{k}{s}a+\frac{s-k}{s}b} ,
\end{aligned}$$
where
$$\begin{aligned}
C'&= 2^{4m-2k+2} \cdot k^{2m},\\
C'' &= 2^{2m+1} \cdot k^{2m+1}.
\end{aligned}$$
\label{prop:conditioning}
\end{prop}
\begin{rem} In due course $H$ and $m$ will be chosen in such a way that we have a power saving in $X$ in the non-well-spaced part. 
\end{rem}

\begin{proof}[Proof of Proposition 5.2] Consider an $\vect{N},\xi$ where the maximum of $I_{a,b}^{g}(X)$ occurs. We now apply Lemma \ref{lem:wellspacing} to $2m$ factors of $|f(\xi,X^{1-b\theta},\vect{\alpha})|^{2(s-k)}$. We find
$$
I^g_{a,b}(X) \le H \left[ \sum_{h = \lceil \log_2(2k) \rceil}^H \frac{2^mL_{h-1}}{m-k} \sum^{(m-k)2^mL_{h-1}} K_{a,b;m}^{g,h}(X) + \frac{L_H}{m} \sum^{m L_H} \CN \right],
$$
where $\CN$ is equal to
\begin{multline*}
\int_{[0,1[^k} |\FF^k_{2^{-g}X^{1-a\theta}}(\vect{N},2^{-g}X^{1-a\theta},\vect{\alpha})|^2 |f(N,2^{-H}X^{1-b\theta},\vect{\alpha})|^{2m} |f(\xi,X^{1-b\theta},\vect{\alpha})|^{2(s-m-k)} d\vect{\alpha}.
\end{multline*}
We refer to the first part as the well-spaced part and second part as the non-well-spaced part.

The well-spaced part is then clearly bounded by
\begin{equation}
H \cdot 2^{4m-2k+2} k^{2m} \sum_{h=\lceil \log_2(2k) \rceil}^H  \left(2^{h}\right)^{2(k-1)} K_{a,b;m}^{g,h}(X)
\label{eq:prop2wsp}
\end{equation}
after inserting the bound $L_{h-1}=2^m k^m (2^{h-1})^{k-1}$. For the non-well-spaced part we bound $\CN$ by H\"older's inequality. This gives the bound
$$
\CN \le \CI_1^{\frac{k}{s}}\CI_2^{\frac{m}{s}}\CI_3^{\frac{s-m-k}{s}},
$$
where
$$\begin{aligned}
\CI_1 &= \int_{[0,1[^k} |\FF^k_{2^{-g}X^{1-a\theta}}(\vect{N},2^{-g}X^{1-a\theta},\vect{\alpha})|^{\frac{2s}{k}} d\vect{\alpha},\\
\CI_2 &= \int_{[0,1[^k} |f(N,2^{-H}X^{1-b\theta},\vect{\alpha})|^{2s} d\vect{\alpha},\\
\CI_3 &= \int_{[0,1[^k} |f(\xi,X^{1-b\theta},\vect{\alpha})|^{2s} d\vect{\alpha}.
\end{aligned}$$
Using AM-GM on
$$
|\FF^k_{2^{-g}X^{1-a\theta}}(\vect{N},2^{-g}X^{1-a\theta},\vect{\alpha})|^{\frac{2s}{k}} \le \frac{1}{k}\sum_{i=1}^k |f(N_i,2^{-g}X^{1-a\theta},\vect{\alpha})|^{2s}
$$
and the integer translation invariance we find that
$$\begin{aligned}
\CI_1 &\le J_{s,k}(2^{-g}X^{1-a\theta}+1),\\
\CI_2 &\le J_{s,k}(2^{-H}X^{1-b\theta}+1),\\
\CI_3 &\le J_{s,k}(X^{1-b\theta}+1).
\end{aligned}$$
Now we have
\begin{equation*}
\begin{aligned}
2^{-g}X^{1-a\theta}+1 &= 2^{-g}X^{1-a\theta}\left(1+2^{g}X^{-1+a\theta} \right)\\
& \le 2^{-g}X^{1-a\theta} \left( 1+ \frac{1}{4\cdot 6^6 k}  \right), \\
\end{aligned}\end{equation*}
since 
$$
2^gX^{-1+a\theta} \le X^{-1+2b\theta} \cdot X^{-\theta} \le X^{-(k^2-2)b\theta}\cdot 2^{-1} \le (2k)^{-(k^2-2)}\cdot 2^{-1} \le \frac{1}{4\cdot 6^6 k}.
$$
Furthermore we have
\begin{equation*}\begin{aligned}
2^{-H}X^{1-b\theta}+1 &= 2^{-H}X^{1-b\theta}\left(1+2^H X^{-1+b\theta}\right) \\
& \le 2^{-H}X^{1-b\theta}\left(1+ \frac{1}{2 \cdot 6^4 k}\right) \\
\end{aligned}\end{equation*}
and
\begin{equation*}\begin{aligned}
X^{1-b\theta}+1 &= X^{1-b\theta}\left(1+X^{-1+b\theta}  \right) \\
& \le X^{1-b\theta}\left(1+ \frac{1}{4 \cdot 6^6k^2}  \right),
\end{aligned}\end{equation*}
since
$$
2^H X^{-1+b\theta} \le X^{-1+(k+1)b\theta} \le X^{-(k^2-k-1)b\theta} \le (2k)^{-(k^2-k-1)} \le \frac{1}{2 \cdot 6^4 k}$$
and
$$
X^{-1+b\theta} \le X^{-(k^2-1)b\theta} \le (2k)^{-(k^2-1)} \le \frac{1}{4 \cdot 6^6k^2}.
$$
Thus we have that the non-well-spaced part is bounded by
\begin{equation}\begin{aligned}
H \cdot L_H^2 \cdot & J_{s,k}\left(\left(1+\frac{1}{4 \cdot 6^6k} \right) 2^{-g}X^{1-a\theta}\right)^{\frac{k}{s}} \\ & \cdot J_{s,k}\left(\left(1+\frac{1}{2 \cdot 6^4 k} \right)2^{-H}X^{1-b\theta}\right)^{\frac{m}{s}} J_{s,k}\left(\left(1+\frac{1}{4 \cdot 6^6k^2} \right)X^{1-b\theta}\right)^{\frac{s-m-k}{s}}.
\label{eq:prop2nwsp}
\end{aligned}\end{equation}

We have now that $I_{a,b}^g(X)$ is bounded by the sum of \eqref{eq:prop2wsp} and \eqref{eq:prop2nwsp}. Taking the maximum and normalising we immediately see that the well-spaced part is true. In the non-well-spaced part we are left with
\begin{equation}\begin{aligned}
H &\cdot L_{H}^2 \cdot \frac{\log \left(2 \left(1+\frac{1}{4 \cdot 6^6k} \right) 2^{-g}X^{1-a\theta} \right)^{\frac{k}{s}\delta}}{\log(2X)^{\frac{k}{s}\delta}} \cdot \frac{\log \left(2 \left(1+\frac{1}{2 \cdot 6^4k} \right) 2^{-H}X^{1-b\theta} \right)^{\frac{m}{s}\delta}}{\log(2X)^{\frac{m}{s}\delta}} \\
&\cdot \frac{\log\left(2 \left(1+\frac{1}{4 \cdot 6^6k^2} \right)X^{1-b\theta}\right)^{\frac{s-m-k}{s}\delta}}{\log(2X)^{\frac{s-m-k}{s}\delta}} \cdot \left( 2^g \right)^{-\frac{k}{s}(2s-\frac{1}{2}k(k+1)+\eta)} \cdot \left(2^H \right)^{-\frac{m}{s}(2s-\frac{1}{2}k(k+1)+\eta)} \\
& \cdot \left(X^{a\theta}\right)^{2k-\frac{1}{2}k(k+1)-\frac{k}{s}(2s-\frac{1}{2}k(k+1)+\eta)} \cdot \left(X^{b\theta} \right)^{2(s-k)-\frac{s-k}{s}(2s-\frac{1}{2}k(k+1)+\eta)} \\
& \cdot \left(1+\frac{1}{4 \cdot 6^6k} \right)^{\frac{k}{s}(2s-\frac{1}{2}k(k+1)+\eta)} \cdot \left(1+\frac{1}{2 \cdot 6^4k} \right)^{\frac{m}{s}(2s-\frac{1}{2}k(k+1)+\eta)} \\
& \cdot \left(1+\frac{1}{4 \cdot 6^6k^2} \right)^{\frac{s-m-k}{s}(2s-\frac{1}{2}k(k+1)+\eta)}.
\end{aligned}
\label{eq:step}
\end{equation}
The $\log$'s are trivially bounded by $1$ again and since $\eta \le \frac{1}{2}k(k+1)$ we have
$$\begin{aligned}
\left(1+\frac{1}{4 \cdot 6^6k} \right)^{\frac{k}{s}(2s-\frac{1}{2}k(k+1)+\eta)} & \le \left(1+\frac{1}{4 \cdot 6^6k} \right)^{2k} \le e^{\frac{1}{2 \cdot 6^6}},\\
\left(1+\frac{1}{2 \cdot 6^4k} \right)^{\frac{m}{s}(2s-\frac{1}{2}k(k+1)+\eta)} & \le \left(1+\frac{1}{2 \cdot 6^4k} \right)^{2m} \le e^{\frac{m}{6^4k}} \le 2e^{-\frac{1}{2 \cdot 6^6}},\\
\left(1+\frac{1}{4 \cdot 6^6k^2} \right)^{\frac{s-m-k}{s}(2s-\frac{1}{2}k(k+1)+\eta)} & \le \left(1+\frac{1}{4 \cdot 6^6k^2} \right)^{2(s-k)} \le e^{\frac{s-k}{2 \cdot 6^6 k^2}} \le k.
\end{aligned}$$
Furthermore we have
\begin{multline*}
\left(X^{a\theta}\right)^{2k-\frac{1}{2}k(k+1)-\frac{k}{s}(2s-\frac{1}{2}k(k+1)+\eta)} \cdot \left(X^{b\theta} \right)^{2(s-k)-\frac{s-k}{s}(2s-\frac{1}{2}k(k+1)+\eta)} \\
=  \left( X^{\theta} \right)^{\frac{1}{2}k(k+1)\frac{s-k}{s}(b-a)} \left( X^{-\eta \theta} \right)^{\frac{k}{s}a+\frac{s-k}{s}b}.
\end{multline*}
Inserting all of these equalities and inequalities together with $L_H=2^mk^m(2^H)^{k-1}$ into \eqref{eq:step} we find that the non-well-spaced part of the proposition is true.\end{proof}

The next proposition is where we extract information as we force diagonal behaviour using Lemma \ref{lem:smallboxes}.

\begin{prop} Let $a,b,m \in \BN_0$ with $b>a \ge 0$ and $6^6 k^2\log(k) \ge s-k \ge m \ge k+1$ and $800k \ge m$. Let $\theta \in \BR$ satisfy $1 \ge k^2b\theta >0$ and let $X \ge 2^{\theta^{-1}}$. Furthermore let $g,h \in \BN$ satisfy $X^{b\theta}\ge 2^g \ge 2k$ and $X^{kb\theta}\ge 2^h \ge 2k$, then we have that $\llbracket K_{a,b;m}^{g,h}(X) \rrbracket$ is bounded by
$$\begin{aligned}
C' \cdot \left( 2^g \right)^{-k+\frac{1}{2}k(k-1)} \left( 2^{h} \right)^{-(2s-\frac{1}{2}k(k+1)+\eta)\left(\frac{m}{s}-\frac{k^2}{s(s-k)}\right)} \cdot \llbracket I_{b,kb}^{h}(X) \rrbracket^{\frac{k}{s-k}} \cdot X^{-\frac{s-2k}{s-k} b\theta \cdot \eta},
\end{aligned}$$
where
$$
C'= 2^{\frac{1}{2}k(k+1)+5} e^{\frac{1}{4}k(3k-2)} k^{-\frac{1}{2}k(k-2)+1} \cdot (s-k+k^2)^k.
$$
\label{prop:extraction}
\end{prop}
\begin{proof} Consider an $\vect{N},\vect{N'},N',\xi$, where the maximum of $K_{a,b;m}^{g,h}(X)$ occurs, and its corresponding diophantine equation:
\begin{equation}
\sum_{i=1}^k (x_i^j-y_i^j) = \sum_{i=1}^k (w_i^j-z_i^j) + \sum_{i=1}^{m-k}(u_i^j-v_i^j) + \sum_{i=1}^{s-m-k} (p_i^j-q_i^j) , \quad j=1,\dots,k,
\label{eq:diopheqK}
\end{equation}
where $\vect{x},\vect{y} \in \FB^k(\vect{N},2^{-g}X^{1-a\theta})$, $\vect{w},\vect{z} \in \FB^k(\vect{N'},2^{-h}X^{1-b\theta})$, $\vect{u},\vect{v} \in \FB^{m-k}(N',2^{-h} X^{1-b \theta})$ and $\vect{w},\vect{z},\vect{u},\vect{v},\vect{p},\vect{q} \in \FB^{s-m-k}(\xi,X^{1-b\theta})$. The right hand side is contained in
\begin{equation}
\left] -2(s-k) \left( 0.5001 \right)^j X^{(1-b\theta)j}, 2(s-k) \left( 0.5001 \right)^j X^{(1-b\theta)j} \right [
\label{eq:smallrhs}
\end{equation}
as
$$\begin{aligned}
\xi + \frac{1}{2} X^{1-b\theta} &\le \frac{1}{2} X^{1-b\theta} \left(1+ X^{-1+b\theta} \right) \\
& \le \frac{1}{2} X^{1-b\theta} \left(1+ X^{-(k^2-1)b\theta} \right)\\
&\le \frac{1}{2} X^{1-b\theta} \left(1+ (2k)^{-(k^2-1)} \right) \\
&\le 0.5001 \cdot X^{1-b\theta}.
\end{aligned}$$
The interval \eqref{eq:smallrhs} we split up into intervals $V_j$ of size at most
$$
(s-k)X^{1-kb\theta} \cdot X^{j-1}.
$$
We have at most
$$\begin{aligned}
\prod_{j=1}^{k} \left( 4 \left( 0.5001 \right)^j X^{(k-j)b\theta} +1 \right) &\le \prod_{j=1}^{\infty} \left(1+4 \left( 0.5001 \right)^j \right) \cdot X^{\frac{1}{2}k(k-1)b\theta} \\
& \le 2^4 \cdot X^{\frac{1}{2}k(k-1)b\theta} = Z', \text{ say},
\end{aligned}$$
choices for $\vect{V}=(V_j)_j$ as
$$
\prod_{j=1}^{\infty} \left(1+4 \left( 0.5001 \right)^j \right) \le \prod_{j=1}^{10} \left(1+4 \left( 0.5001 \right)^j \right) \cdot \exp \! \! \left(4\sum_{j=11}^{\infty} \left( 0.5001 \right)^j\right)<14.27 \cdot 1.004 < 2^4.
$$
Furthermore we split up the box $\FB^k(\vect{N},2^{-g} X^{1-a\theta})$ for the $\vect{y}$'s into sub-boxes of the shape $\FB^k(\vect{N''},X^{1-kb\theta})$. We have at most
$$\begin{aligned}
\left( \frac{X^{(kb-a)\theta}}{2^g}+1 \right)^k & = \left(1+ 2^gX^{-(kb-a)\theta} \right)^k\frac{X^{(kb-a)k\theta}}{2^{gk}}\\
& \le e^{\frac{1}{4}} \frac{X^{(kb-a)k\theta}}{2^{gk}} = Z''_g, \text{ say},
\end{aligned}$$
of these, since
$$
2^gX^{-(kb-a)\theta} \le X^{-((k-1)b-a)\theta}\le X^{-(k-2)b\theta}X^{-\theta} \le (2k)^{-(k-2)}\cdot 2^{-1}
$$
and
$$
\left(1+(2k)^{-(k-2)}\cdot 2^{-1}\right)^k \le e^{\frac{1}{4(2k)^{k-3}}} \le e^{\frac{1}{4}}.
$$
Let $\CS(\FB^k(\vect{N''},X^{1-kb\theta}),\vect{V})$ denote the number of solutions $(\vect{x},\vect{y},\vect{w},\vect{z},\vect{u},\vect{v},\vect{p},\vect{q})$ of \eqref{eq:diopheqK} with the additional restriction that
$$
\sum_{i=1}^k(x_i^j-y_i^j) \in V_j \quad (j=1,\dots,k)
$$
and $\vect{y} \in \FB^k(\vect{N''},X^{1-kb\theta})$, so that the total number of solutions to \eqref{eq:diopheqK} is bounded by
\begin{equation}
\sum^{Z'}\sum^{Z''_g} \CS(\FB^k(\vect{N''},X^{1-kb\theta}),\vect{V}).
\label{eq:betwbound}
\end{equation}
Two solutions $(\vect{x},\vect{y},\vect{w},\vect{z},\vect{u},\vect{v},\vect{p},\vect{q}),(\vect{x'},\vect{y'},\vect{w'},\vect{z'},\vect{u'},\vect{v'},\vect{p'},\vect{q'})$ of $\CS(\FB^k(\vect{N'},X^{1-kb\theta}),\vect{V})$ satisfy the inequality
\begin{equation}\begin{aligned}
\left|\sum_{i=1}^k x_i^j - \sum_{i=1}^k x_i'^j \right| & \le \left| \sum_{i=1}^k (x_i^j-y_i^j) - \sum_{i=1}^k (x_i'^j-y_i'^j) \right| + \left| \sum_{i=1}^k y_i^j - \sum_{i=1}^k y_i'^j \right| \\
& \le (s-k)X^{1-kb\theta} \cdot X^{j-1}+jkX^{1-kb\theta} \cdot X^{j-1}\\
& \le (s-k+k^2) X^{1-kb\theta} \cdot X^{j-1}.
\end{aligned}
\label{eq:uint1}
\end{equation}
Thus we are able to apply Lemma \ref{lem:smallboxes} to bound $\CS(\FB^k(\vect{N''},X^{1-kb\theta}),\vect{V})$. We are able to apply it with $S=(s-k+k^2)X^{1-kb\theta}$, $R=2^{-g}X^{1-a\theta}$, $\vect{U}$ the interval in \eqref{eq:uint1}, and $W(\vect{x})$ being the number of solutions $(\vect{x},\vect{y},\vect{w},\vect{z},\vect{u},\vect{v},\vect{p},\vect{q})$ counted by $\CS(\FB^k(\vect{N''},X^{1-kb\theta}),\vect{V})$. We get that \eqref{eq:betwbound} is bounded by
\begin{equation}
\sum^{Z'}\sum^{Z''_g} 2^{\frac{1}{2}k(k+1)} e^{\frac{1}{4}(3k+1)(k-1)} k^{-\frac{1}{2}k(k-2)} \cdot \left(2^g X^{a\theta}  \right)^{\frac{1}{2}k(k-1)} \cdot  \CZ_W(\FB^k(\vect{N'''},\vect{S'}),\vect{U}),
\label{eq:betwbound2}
\end{equation}
with $1 \le \vect{S'} \le (s-k+k^2)X^{1-kb\theta}$. Now $\CZ_W(\FB^k(\vect{N'''},\vect{S'}),\vect{U})$ is just counting the number of solutions of \eqref{eq:diopheqK} with some further restrictions. The two we care about are $\vect{x}\in \FB^k(\vect{N'''},\vect{S'})$ and $\vect{y}\in \FB^k(\vect{N''},X^{1-kb\theta})$. Therefore we have
\begin{multline}
\CZ_W(\FB^k(\vect{N'''},\vect{S'}),\vect{U}) \\ \le  \int_{[0,1[^k} \FF^k_{2^{-g}X^{1-a\theta}}(\vect{N'''},\vect{S'},\vect{\alpha})\FF^k_{2^{-g}X^{1-a\theta}}(\vect{N''},X^{1-kb\theta},-\vect{\alpha}) \cdot \Ff^{\star} d\vect{\alpha},
\label{eq:intermedstep}
\end{multline}
where
$$\begin{aligned}
\Ff^{\star} &= |\FF^k_{2^{-h}X^{1-b\theta}}(\vect{N'},2^{-h} X^{1-b\theta},\vect{\alpha})|^2 |f(N',2^{-h} X^{1-b\theta},\vect{\alpha})|^{2(m-k)}|f(\xi,X^{1-b\theta},\vect{\alpha})|^{2(s-m-k)}.
\end{aligned}$$
We split up $\FB^k(\vect{N'''},\vect{S'})$ further into $(s-k+k^2)^k$ sub-boxes of size at most $X^{1-kb\theta}$, which we may assume to have exactly size $X^{1-kb\theta}$. Thus the integral in \eqref{eq:intermedstep} is further bounded by
$$
\sum^{(s-k+k^2)^k} \int_{[0,1[^k} \FF^k_{2^{-g}X^{1-a\theta}}(\vect{N''''},X^{1-kb\theta},\vect{\alpha})\FF^k_{2^{-g}X^{1-a\theta}}(\vect{N''},X^{1-kb\theta},-\vect{\alpha}) \cdot \Ff^{\star} d\vect{\alpha}.
$$
Using H\"older's inequality we further find 
\begin{equation}
\CZ_W(\FB^k(\vect{N'''},\vect{S'}),\vect{U}) \le \sum^{(s-k+k^2)^k} \CI_1^{\frac{k}{2(s-k)}}\CI_2^{\frac{k}{2(s-k)}}\CI_3^{\frac{(s-2k)k}{s(s-k)}}\CI_4^{\frac{m-k}{s}}\CI_5^{\frac{s-m-k}{s}},
\label{eq:holdering}
\end{equation}
where
$$\begin{aligned}
\CI_1 &= \int_{[0,1[^k} |\FF^k_{2^{-h}X^{1-b\theta}}(\vect{N'},2^{-h} X^{1-b\theta},\vect{\alpha})|^2 |\FF^k_{2^{-g}X^{1-a\theta}}(\vect{N''''},X^{1-kb\theta},\vect{\alpha})|^{\frac{2(s-k)}{k}} d\vect{\alpha}, \\
\CI_2 &= \int_{[0,1[^k} |\FF^k_{2^{-h}X^{1-b\theta}}(\vect{N'},2^{-h} X^{1-b\theta},\vect{\alpha})|^2 |\FF^k_{2^{-g}X^{1-a\theta}}(\vect{N''},X^{1-kb\theta},\vect{\alpha})|^{\frac{2(s-k)}{k}} d\vect{\alpha}, \\
\CI_3 &= \int_{[0,1[^k} |\FF^k_{2^{-h}X^{1-b\theta}}(\vect{N'},2^{-h} X^{1-b\theta},\vect{\alpha})|^{\frac{2s}{k}} d\vect{\alpha}, \\
\CI_4 &= \int_{[0,1[^k} |f(N', 2^{-h} X^{1-b\theta},\vect{\alpha})|^{2s}  d\vect{\alpha}, \\
\CI_5 &= \int_{[0,1[^k} |f(\xi, X^{1-b\theta},\vect{\alpha})|^{2s}  d\vect{\alpha}.
\end{aligned}$$
Using AM-GM on
$$
|\FF^k_{2^{-g}X^{1-a\theta}}(\vect{N''},X^{1-kb\theta},\vect{\alpha})|^{\frac{2(s-k)}{k}} \le \frac{1}{k}\sum_{i=1}^k |f(N''_i,X^{1-kb\theta},\vect{\alpha})|^{2(s-k)}
$$
and the integer translation invariance we find that
$$
\CI_1 \le I_{b,bk}^{h}(X).
$$
Analogously also
$$
\CI_2 \le I_{b,bk}^{h}(X)
$$
holds. Using AM-GM in a similar fashion we find by the integer translation invariance that
$$
\CI_3,\CI_4 \le J_{s,k}\left(2^{-h} X^{1-b\theta} + 1 \right).
$$
Again by the integer translation invariance we find that
$$
\CI_5 \le J_{s,k}\left( X^{1-b\theta}+1 \right).
$$
We have
\begin{equation*}\begin{aligned}
2^{-h} X^{1-b\theta} + 1 &= 2^{-h} X^{1-b\theta} \left(1+ 2^hX^{-1+b\theta} \right)\\
& \le 2^{-h} X^{1-b\theta} \left(1+ \frac{1}{2 \cdot 6^4k} \right),
\end{aligned}\end{equation*}
since
$$
2^hX^{-1+b\theta} \le X^{-1+(k+1)b\theta} \le X^{-(k^2-k-1)b\theta} \le (2k)^{-(k^2-k-1)} \le \frac{1}{2 \cdot 6^4k},
$$
and
\begin{equation*}\begin{aligned}
X^{1-b\theta}+1 &= X^{1-b\theta}\left(1+X^{-1+b\theta} \right) \\
& \le X^{1-b\theta}\left(1+\frac{1}{4 \cdot 6^6k^2} \right),
\end{aligned}\end{equation*}
since
$$
X^{-1+b\theta} \le X^{-(k^2-1)b\theta} \le (2k)^{-(k^2-1)} \le \frac{1}{4 \cdot 6^6k^2}.
$$

By inserting the above analysis into \eqref{eq:holdering} and further into \eqref{eq:betwbound2} we conclude that $K_{a,b;m}^{g,h}(X)$ is bounded by
\begin{multline*}
2^4 X^{\frac{1}{2}k(k-1)b\theta} e^{\frac{1}{4}}2^{-gk} X^{(kb-a)k\theta} 2^{\frac{1}{2}k(k+1)} e^{\frac{1}{4}(3k+1)(k-1)} k^{-\frac{1}{2}k(k-2)} \cdot \left(2^g X^{a\theta}  \right)^{\frac{1}{2}k(k-1)} (s-k+k^2)^k \\ \cdot I_{b,bk}^{h}(X)^{\frac{k}{s-k}} J_{s,k} \left( \! \left(1+\frac{1}{2 \cdot 6^4k} \right)2^{-h}X^{1-b\theta} \right)^{\frac{m}{s}-\frac{k^2}{s(s-k)}} J_{s,k} \left( \! \left(1+\frac{1}{4 \cdot 6^6k^2} \right)X^{1-b\theta} \right)^{\frac{s-m-k}{s}}.
\end{multline*}
Let us apply the normalisations and analyse each parameter separately. The dependence on $X$ is going to be
$$\begin{aligned}
& \frac{\log \left( 2\left(1+\frac{1}{2\cdot 6^4k} \right)2^{-h}X^{1-b\theta} \right)^{\delta\left( \frac{m}{s}-\frac{k^2}{s(s-k)}\right)}}{\log(2X)^{\delta\left( \frac{m}{s}-\frac{k^2}{s(s-k)}\right)}} \cdot \frac{\log \left( 2\left(1+\frac{1}{4\cdot 6^6k^2} \right)X^{1-b\theta} \right)^{\delta\frac{s-m-k}{s}}}{\log(2X)^{\delta\frac{s-m-k}{s}}}\\
&  \cdot \left( X^{a\theta} \right)^{2k-\frac{1}{2}k(k+1)-k+\frac{1}{2}k(k-1)} \left(X^{b\theta} \right)^{2(s-k)+\frac{1}{2}k(k-1)+k^2-\frac{k}{s-k}\left(2k-\frac{1}{2}k(k+1)+2(s-k)k \right)}\\
&  \cdot (X^{b\theta})^{-\frac{s-2k}{s-k}\left(2s-\frac{1}{2}k(k+1)+\eta\right)}.
\end{aligned}$$
The fraction with $\log$'s are bounded by $1$ again. The exponent of $X^{a\theta}$ is $0$ and the exponent of $X^{b\theta}$ reduces to $-\frac{s-2k}{s-k}\eta$ after a short computation. The dependence on $h$ is
$$
\left(2^h\right)^{-\left(2s-\frac{1}{2}k(k+1)+\eta \right)\left( \frac{m}{s}-\frac{k^2}{s(s-k)} \right)}.
$$
The dependence on $g$ is
$$
\left(2^{g}\right)^{-k +\frac{1}{2}k(k-1)}.
$$
And finally the constant is
$$\begin{aligned}
2^4 \cdot& e^{\frac{1}{4}} \cdot 2^{\frac{1}{2}k(k+1)} e^{\frac{1}{4}(3k+1)(k-1)} \cdot k^{-\frac{1}{2}k(k-2)} \cdot (s-k+k^2)^k \\
\cdot& \left(1+\frac{1}{2 \cdot 6^4k} \right)^{\left(2s-\frac{1}{2}k(k+1)+\eta \right)\left(\frac{m}{s}-\frac{k^2}{s(s-k)} \right)} \cdot \left(1+\frac{1}{4 \cdot 6^6k^2} \right)^{\left(2s-\frac{1}{2}k(k+1)+\eta \right)\frac{s-m-k}{s}}.
\end{aligned}$$
Since $\eta \le \frac{1}{2}k(k+1)$ we have
$$
\left(1+\frac{1}{2 \cdot 6^4k} \right)^{\left(2s-\frac{1}{2}k(k+1)+\eta \right)\left(\frac{m}{s}-\frac{k^2}{s(s-k)} \right)} \le e^{\frac{2s}{2\cdot 6^4k}\frac{m}{s}} \le e^{\frac{m}{6^4k}} \le 2
$$
and
$$
\left(1+\frac{1}{4 \cdot 6^6k^2} \right)^{\left(2s-\frac{1}{2}k(k+1)+\eta \right)\frac{s-m-k}{s}} \le e^{\frac{2s}{4 \cdot 6^6k^2} \frac{s-k}{s}} \le e^{\frac{s-k}{2 \cdot 6^6k^2}} \le k.
$$
Thus we find that the constant is bounded by
$$
2^{\frac{1}{2}k(k+1)+5} e^{\frac{1}{4}k(3k-2)} k^{-\frac{1}{2}k(k-2)+1} \cdot (s-k+k^2)^k.
$$
\end{proof}

This last proposition is essentially the same as the previous one, the difference being that the iteration comes to a halt after this step.

\begin{prop} Let $a,b \in \BN_0$ with $b>a$. Let $\theta \in \BR$ satisfy $1 \ge kb\theta > 0$ and let $X \ge 2^{\theta^{-1}}$. Furthermore let $g \in \BN$ satisfy $X^{b\theta}\ge 2^g \ge 2k$and $2 k^2 \log(k) \ge s-k$. Then we have:
$$
\llbracket I_{a,b}^{g}(X) \rrbracket \le C' \cdot \left( 2^g \right)^{-k+\frac{1}{2}k(k-1)} X^{\frac{k^2(k^2-1)}{2s}b\theta} X^{-\eta\frac{s+k^2-k}{s}b\theta},
$$
where
$$
C'=2^{\frac{1}{2}k(k+5)+4} \cdot e^{\frac{1}{4}k(3k-2)} \cdot k^{-\frac{1}{2}k(k-2)+1} \cdot (s-k+k^2)^k.
$$
\label{prop:endI}
\end{prop}
\begin{proof} Consider an $\vect{N},\xi$, where the maximum of $I_{a,b}^{g}(X)$ occurs, and its corresponding diophantine equation:
\begin{equation}
\sum_{i=1}^k (x_i^j-y_i^j) = \sum_{i=1}^{s-k} (p_i^j-q_i^j) , \quad j=1,\dots,k,
\label{eq:diopheqI}
\end{equation}
where $\vect{x},\vect{y} \in \FB^k(\vect{N},2^{-g}X^{1-a\theta})$ and $\vect{p},\vect{q} \in \FB^{s-k}( \xi,X^{1-b\theta})$.

From here we proceed as in the previous proposition, but in this case we need to adjust our definition of $\CS(\FB^k(\vect{N''},X^{1-kb\theta}),\vect{V})$. Let $\CS(\FB^k(\vect{N''},X^{1-kb\theta}),\vect{V})$ denote the number of solutions $(\vect{x},\vect{y},\vect{p},\vect{q})$ of \eqref{eq:diopheqI} with the additional restriction that
$$
\sum_{i=1}^k(x_i^j-y_i^j) \in V_j \quad (j=1,\dots,k)
$$
and $\vect{y} \in \FB^k(\vect{N''},X^{1-kb\theta})$, so that the total number of solutions to \eqref{eq:diopheqI} is bounded by
\begin{equation}
\sum^{Z'}\sum^{Z''_g} \CS(\FB^k(\vect{N''},X^{1-kb\theta}),\vect{V}).
\label{eq:betwboundI}
\end{equation}
Consider now two solutions $(\vect{x},\vect{y},\vect{p},\vect{q}),(\vect{x'},\vect{y'},\vect{p'},\vect{q'})$ of $\CS(\FB^k(\vect{N'},X^{1-kb\theta}),\vect{V})$. In this case we have
\begin{equation}
\begin{aligned}
\left|\sum_{i=1}^k x_i^j - \sum_{i=1}^k x_i'^j \right| & \le \left| \sum_{i=1}^k (x_i^j-y_i^j) - \sum_{i=1}^k (x_i'^j-y_i'^j) \right| + \left| \sum_{i=1}^k y_i^j - \sum_{i=1}^k y_i'^j \right| \\
& \le (s-k)X^{1-kb\theta} \cdot X^{j-1}+jkX^{1-kb\theta} \cdot X^{j-1}\\
& \le (s-k+k^2) X^{1-kb\theta} \cdot X^{j-1}.
\end{aligned}
\label{eq:uint2}
\end{equation}
Thus we are again able to apply Lemma \ref{lem:smallboxes} to bound $\CS(\FB^k(\vect{N''},X^{1-kb\theta}),\vect{V})$. This time with $W(\vect{x})$ being the number of solutions $(\vect{x},\vect{y},\vect{p},\vect{q})$ counted by $\CS(\FB^k(\vect{N''},X^{1-kb\theta}),\vect{V})$ and $\vect{U}$ being the interval in \eqref{eq:uint2}. We arrive at the conclusion that \eqref{eq:betwboundI} is bounded by
\begin{equation}
\sum^{Z'}\sum^{Z_g''}2^{\frac{1}{2}k(k+1)} e^{\frac{1}{4}(3k+1)(k-1)} k^{-\frac{1}{2}k(k-2)} \cdot \left(2^g X^{a\theta}  \right)^{\frac{1}{2}k(k-1)} \cdot \CZ_W(\FB^k(\vect{N'''},\vect{S'}),\vect{U}),
\label{eq:sbound}
\end{equation}
with $1 \le \vect{S'} \le (s-k+k^2)X^{1-kb\theta}$. Now $\CZ_W(\FB^k(\vect{N'''},\vect{S'}),\vect{U})$ is just counting the number of solutions of \eqref{eq:diopheqI} with some further restrictions. The two we care about are $\vect{x}\in \FB^k(\vect{N'''},\vect{S'})$ and $\vect{y}\in \FB^k(\vect{N''},X^{1-kb\theta})$. Thus we arrive at
\begin{multline}
\CZ_W(\FB^k(\vect{N'''},\vect{S'}),\vect{U}) \\ \le \int_{[0,1[^k} \FF^k_{2^{-g}X^{1-a\theta}}(\vect{N'''},\vect{S'},\vect{\alpha})\FF^k_{2^{-g}X^{1-a\theta}}(\vect{N''},X^{1-kb\theta},-\vect{\alpha}) |f(\xi,X^{1-b\theta},\vect{\alpha})|^{2(s-k)} d\vect{\alpha}.
\label{eq:intermedstepI}
\end{multline}
We split up $\FB^k(\vect{N'''},\vect{S'})$ further into $(s-k+k^2)^k$ sub-boxes of size at most $X^{1-kb\theta}$, which we may assume to have exactly size $X^{1-kb\theta}$. Thus the integral in \eqref{eq:intermedstepI} is further bounded by
$$\begin{aligned}
\sum^{(s-k+k^2)^k} \int_{[0,1[^k} \FF^k_{2^{-g}X^{1-a\theta}}(\vect{N''''},X^{1-kb\theta},\vect{\alpha})\FF^k_{2^{-g}X^{1-a\theta}}(\vect{N''}&,X^{1-kb\theta},-\vect{\alpha}) \\
& \cdot |f(\xi,X^{1-b\theta},\vect{\alpha})|^{2(s-k)}  d\vect{\alpha}.
\end{aligned}$$
Further using H\"older's inequality we find
\begin{equation}
\CZ_W(\FB^k(\vect{N'''},\vect{S'}),\vect{U}) \le \sum^{(s-k+k^2)^k}\CI_1^{\frac{k}{2s}}\CI_2^{\frac{k}{2s}}\CI_3^{\frac{s-k}{s}},
\label{eq:Zbound}
\end{equation}
where
$$\begin{aligned}
\CI_1 &= \int_{[0,1[^k} |\FF^k_{2^{-g}X^{1-a\theta}}(\vect{N''},X^{1-kb\theta},\vect{\alpha})|^{\frac{2s}{k}} d\vect{\alpha}, \\
\CI_2 &= \int_{[0,1[^k} |\FF^k_{2^{-g}X^{1-a\theta}}(\vect{N},X^{1-kb\theta},\vect{\alpha})|^{\frac{2s}{k}} d\vect{\alpha}, \\
\CI_3 &= \int_{[0,1[^k} |f(\xi, X^{1-b\theta},\vect{\alpha})|^{2s}  d\vect{\alpha}.
\end{aligned}$$
Using AM-GM on
$$
|\FF^k_{2^{-g}X^{1-a\theta}}(\vect{N''},X^{1-kb\theta},\vect{\alpha})|^{\frac{2s}{k}} \le \frac{1}{k}\sum_{i=1}^k |f(N''_i,X^{1-kb\theta},\vect{\alpha})|^{2s}
$$
and the integer translation invariance we find that
$$
\CI_1 \le J_{s,k}(X^{1-kb\theta}+1)\le J_{s,k}(2X^{1-kb\theta}).
$$
Analogously also
$$
\CI_2 \le J_{s,k}(X^{1-kb\theta}+1) \le J_{s,k}(2X^{1-kb\theta})
$$
holds. And finally by the integer translation invariance we find
$$
\CI_3 \le J_{s,k}\left( X^{1-b\theta}+1 \right).
$$
We have
$$\begin{aligned}
X^{1-b\theta}+1 &=X^{1-b\theta}\left(1+X^{-1+b\theta} \right)\\
& \le X^{1-b\theta}\left(1+\frac{1}{4 k^2} \right),
\end{aligned}$$
since
$$
X^{-1+b\theta} \le X^{-(k-1)b\theta} \le (2k)^{-(k-1)} \le \frac{1}{4 k^2}.
$$
Inserting the above analysis into \eqref{eq:Zbound} and further \eqref{eq:sbound} we conclude that $I_{a,b}^{g}(X)$ is bounded by
\begin{multline*}
2^4  X^{\frac{1}{2}k(k-1)b\theta} e^{\frac{1}{4}}2^{-gk} X^{(kb-a)k\theta}  2^{\frac{1}{2}k(k+1)} e^{\frac{1}{4}(3k+1)(k-1)} k^{-\frac{1}{2}k(k-2)} \cdot \left(2^g X^{a\theta}  \right)^{\frac{1}{2}k(k-1)} \\ \cdot (s-k+k^2)^k
\cdot  J_{s,k}(2X^{1-kb\theta})^{\frac{k}{s}} J_{s,k} \left( \left(1+\frac{1}{4 k^2} \right)X^{1-b\theta} \right)^{\frac{s-k}{s}}.
\end{multline*}
Let us apply the normalisations and analyse each parameter separately. The dependence on $X$ is going to be
$$\begin{aligned}
& \frac{\log \left( 4X^{1-kb\theta} \right)^{\delta\frac{k}{s}}}{\log(2X)^{\delta\frac{k}{s}}} \cdot \frac{\log \left( 2\left(1+\frac{1}{4k^2} \right)X^{1-b\theta} \right)^{\delta\frac{s-k}{s}}}{\log(2X)^{\delta\frac{s-k}{s}}}  \cdot \left( X^{a\theta} \right)^{2k-\frac{1}{2}k(k+1)-k+\frac{1}{2}k(k-1)} \\
& \cdot \left(X^{b\theta} \right)^{2(s-k)+\frac{1}{2}k(k-1)+k^2-\frac{k^2}{s}\left(2s-\frac{1}{2}k(k+1)+\eta \right)-\frac{s-k}{s}\left(2s-\frac{1}{2}k(k+1)+\eta\right)}.
\end{aligned}$$
The fraction with $\log$'s are bounded by $1$ again. The exponent of $X^{a\theta}$ is $0$ and the exponent of $X^{b\theta}$ reduces to $\frac{k^2(k^2-1)}{2s}-\frac{s-k+k^2}{s}\eta$ after a short computation. The dependence on $g$ is $\left(2^{g}\right)^{-k +\frac{1}{2}k(k-1)}$.
And finally the constant is
\begin{multline*}
2^4 \cdot e^{\frac{1}{4}} \cdot 2^{\frac{1}{2}k(k+1)} e^{\frac{1}{4}(3k+1)(k-1)} \cdot k^{-\frac{1}{2}k(k-2)} \cdot (s-k+k^2)^k \\
\cdot 2^{\left(2s-\frac{1}{2}k(k+1)+\eta \right)\frac{k}{s}} \cdot \left(1+\frac{1}{4 k^2} \right)^{\left(2s-\frac{1}{2}k(k+1)+\eta \right)\frac{s-k}{s}}.
\end{multline*}
Since $\eta \le \frac{1}{2}k(k+1)$ we have
$$
2^{\left(2s-\frac{1}{2}k(k+1)+\eta \right)\frac{k}{s}} \le 2^{2k}
$$
and
$$
\left(1+\frac{1}{4 k^2} \right)^{\left(2s-\frac{1}{2}k(k+1)+\eta \right)\frac{s-k}{s}} \le e^{\frac{2s}{4k^2} \frac{s-k}{s}} \le e^{\frac{s-k}{2k^2}} \le k.
$$
Therefore we see that the constant is bounded by
$$
2^4 \cdot 2^{\frac{1}{2}k(k+1)} e^{\frac{1}{4}k(3k-2)} k^{-\frac{1}{2}k(k-2)} \cdot (s-k+k^2)^k \cdot  2^{2k} \cdot k.
$$

\end{proof}

\section{Iterative Process}
In this section we iterate through the Propositions \ref{prop:conditioning} and \ref{prop:extraction} as often as we can. This was already outlined in Section \ref{sec:outline} and we recommend the reader to have a second look at it before advancing, since the argument to follow is essentially the same with the exception that there are more parameters to be analysed and chosen.\\

Let us recall some of our notation of the outline. Let $D \ge 1$ be an integer and set $\theta=k^{-(D+1)}$. Let $(a_0,b_0),(a_1,b_1),(a_2,b_2),\dots,(a_D,b_D)$ denote the sequence
$$
(0,1),(1,k),(k,k^2),\dots,(k^{D-1},k^D).
$$
Furthermore we assume $X\ge 2^{k^{D+1}}$ and $2 \log(k) \ge \lambda = \frac{s-k}{k^2} \ge 1$. We now fix a choice of parameters, which we will justify later on. Set
\begin{equation*}
G_n= \lfloor k^n \theta \log_2(X) \rfloor, \quad  \text{for } n=0,\dots,D
\end{equation*}
and
\begin{equation}
m_n = \begin{cases} \left \lfloor \frac{1}{4}k(k+1) + \frac{4}{3}k - \frac{1}{2} \right \rfloor, & \text{if } n=0,\\
\left \lfloor \frac{5}{3}k \right \rfloor, & \text{if } n=1,\dots,D. \end{cases}
\label{eq:m}
\end{equation}
We remark here that the choice of $G_n$ will ensure that the conditions
\begin{equation}
X^{kb_{n-1}\theta}=X^{b_n\theta} \ge 2^{G_n} \ge 2^{g_n} \ge 2k
\label{eq:Gineq1}
\end{equation}
of the Propositions \ref{prop:initial}, \ref{prop:conditioning}, \ref{prop:extraction} and \ref{prop:endI} are satisfied, where the last inequality comes from the restriction of our well-spaced parameter $g_n$ in Lemma \ref{lem:wellspacing}. We would also like to highlight the inequalities
\begin{equation*}
\frac{1}{4}k(k+1)+\frac{4}{3}k-\frac{1}{2} \ge m_0 \ge \frac{1}{4}k(k+1)+\frac{4}{3}k-\frac{4}{3}
\end{equation*}
and
\begin{equation*}
\frac{5}{3}k \ge m_n \ge \frac{5}{3}k-\frac{2}{3}, \quad \forall n=1,\dots,D,
\end{equation*}
which will be frequently used. The conditions of Proposition \ref{prop:initial} are now clearly met, thus we get

$$\begin{aligned}
\llbracket J_{s,k}(X) \rrbracket \le & C' \cdot G_0 \sum_{g_0= \lceil \log_2(2k) \rceil}^{G_0}  \left(2^{g_0} \right)^{2(k-1)-2(m_0-k)}\llbracket I_{0,1}^{g_0}(X) \rrbracket\\
 & + C'' \cdot G_0  \left(2^{G_0}\right)^{2(k-1)-\frac{m_0}{s}(2s-\frac{1}{2}k(k+1)+\eta)},
\end{aligned}$$
where
$$\begin{aligned}
C' &= 2^{6m_0-4k+2}k^{2m_0} \cdot \left(1+\frac{1}{X^{\theta}} \right)^{2(s-m_0)},\\
C'' &= 2^{2m_0+1}k^{2m_0}.
\end{aligned}$$
Using the inequalities on $m_0$ and $G_0 \le k^{-(D+1)} \log_2(2X)$ we find
$$
\llbracket J_{s,k}(X) \rrbracket \le \log_2(2X) \left( C_0 \sum_{g_0= \lceil \log_2(2k) \rceil}^{G_0}  \left(2^{g_0} \right)^{2(k-1)-2(m_0-k)}\llbracket I_{0,1}^{g_0}(X) \rrbracket + E_0  \right),
$$
where
$$
C_0= 2^{\frac{3}{2}k^2+\frac{11}{2}k-1} k^{\frac{1}{2}k^2+\frac{19}{6}k-2-D} \left(1+\frac{1}{X^{\theta}} \right)^{2(s-m_0)}
$$
and
$$\begin{aligned}
E_0&=2^{\frac{1}{2}k^2+\frac{19}{6}k}k^{\frac{1}{2}k^2+\frac{19}{6}k-2-D} \left( 2^{G_0} \right)^{2(k-1)-\frac{m_0}{s}(2s-\frac{1}{2}k(k+1)+\eta)}.
\end{aligned}$$
We have
$$
\frac{2s-\frac{1}{2}k(k+1)+\eta}{s} \ge \frac{3}{2}+\frac{\eta}{k(k+1)} \Leftrightarrow \left( \frac{1}{2}k(k+1)-\eta \right) \left( \frac{s}{k(k+1)}-1 \right) \ge 0
$$
and therefore we further find
$$\begin{aligned}
E_0& \le 2^{\frac{1}{2}k^2+\frac{19}{6}k}k^{\frac{1}{2}k^2+\frac{19}{6}k-2-D} \left( 2^{G_0} \right)^{2(k-1)-\frac{3}{2}m_0-\frac{m_0}{k(k+1)}\eta} \\
& \le 2^{\frac{1}{2}k^2+\frac{19}{6}k}k^{\frac{1}{2}k^2+\frac{19}{6}k-2-D} \left( 2^{G_0} \right)^{-\frac{3}{8}k(k+1)-\frac{1}{4}\eta} \\
& \le 2^{\frac{1}{2}k^2+\frac{19}{6}k}k^{\frac{1}{2}k^2+\frac{19}{6}k-2-D} \left( \frac{X^{\theta}}{2} \right)^{-\frac{3}{8}k(k+1)-\frac{1}{4}\eta} \\
& \le 2^{k^2+\frac{11}{3}k}k^{\frac{1}{2}k^2+\frac{19}{6}k-2-D} X^{-\eta \theta}.
\end{aligned}$$
In conclusion we have
\begin{equation*}
\llbracket J_{s,k}(X) \rrbracket \le  \log_2(2X) \cdot \Psi_0,
\end{equation*}
where
\begin{equation*}\begin{aligned}
\Psi_0 &=\CC_0 \sum_{g_0= \lceil \log_2(2k) \rceil}^{G_0} \left( 2^{g_0} \right)^{\alpha_0} \llbracket I_{a_0,b_0}^{g_0}(X) \rrbracket + \CC_0^{\dagger} \cdot X^{-\eta \theta\frac{s-2k}{s-k}}
\end{aligned}\end{equation*}
and
$$\begin{aligned}
\CC_0 &= 2^{\frac{3}{2}k^2+\frac{11}{2}k-1} k^{\frac{1}{2}k^2+\frac{19}{6}k-2-D} \left(1+\frac{1}{X^{\theta}} \right)^{2(s-m_0)}, \\
\CC_0^{\dagger} &= 2^{k^2+\frac{11}{3}k}k^{\frac{1}{2}k^2+\frac{19}{6}k-2-D},\\
\alpha_0 &= 2(k-1)-2(m_0-k).
\end{aligned}$$
It is evident that we gave up some saving in the error term $E_0$. This is because this is the maximal amount of power saving we are able get in the error term $E_1$ of the next iteration.\\

We further define
\begin{equation}\begin{aligned}
\Psi_n =& \CC_n  \left(X^{-\eta \theta} \right)^{\frac{s-2k}{s-k} \sum_{i=0}^{n-1}b_i \left( \frac{k}{s-k} \right)^{i} } \left( \sum_{g_n= \lceil \log_2(2k) \rceil}^{G_n} \left( 2^{g_n} \right)^{\alpha_n} \llbracket I_{a_n,b_n}^{g_n}(X) \rrbracket^{\frac{k}{s-k}} \right)^{\left( \frac{k}{s-k} \right)^{n-1}} \\
&+ \CC_n^{\dagger} \cdot X^{-\eta \theta\frac{s-2k}{s-k}},
\label{eq:psidef}
\end{aligned}\end{equation}
for $n=1,\dots,D$, where
$$
\alpha_n= \begin{cases} 2(k-1)-2(m_0-k), & n=0,\\2(k-1)-(2s-\frac{1}{2}k(k+1)+\eta)\left(\frac{m_n}{s}-\frac{k^2}{s(s-k)}\right), & n= 1,\dots,D, \end{cases}
$$
and $\CC_n,\CC_n^{\dagger}$ are some constants, which are going to be defined recursively in \eqref{eq:c0} and \eqref{eq:cn}. We now use Propositions \ref{prop:conditioning} and \ref{prop:extraction} to prove the following proposition.
\begin{prop} With the notation as above and the assumptions mentioned at the beginning of this section we have
\begin{equation}
\Psi_n \le \log_2(2X)^{\left(\frac{k}{s-k}\right)^n} \cdot \Psi_{n+1}, \quad \forall n=0,\dots,D-1.
\label{eq:ineqpsi}
\end{equation}
\label{prop:ineqpsi}
\end{prop}
\begin{proof} As the cases $n=0$ and $n\ge 1$ are quite similar we will consider them at the same time. Because of the Inequality \eqref{eq:Gineq1} and because $n\le D-1$ implies $1 \ge k^2b_n \theta > 0$ we are able to apply Proposition \ref{prop:conditioning} to $\llbracket I_{a_n,b_n}^{g_n}(X) \rrbracket$ and get
$$
\llbracket I_{a_n,b_n}^{g_n}(X) \rrbracket \le C_{n+1} \cdot G_{n+1} \sum_{g_{n+1}=\lceil \log_2(2k) \rceil}^{G_{n+1}} \left(2^{g_{n+1}}\right)^{2(k-1)} \llbracket K_{a_n,b_n;m_{n+1}}^{g_n,g_{n+1}}(X) \rrbracket+G_{n+1}\cdot E_{n+1},
$$
where
\begin{equation}
C_{n+1} = 2^{4m_{n+1}-2k+2} \cdot k^{2m_{n+1}}
\label{eq:Cdefi}
\end{equation}
and
\begin{equation}\begin{aligned}
E_{n+1}=&2^{2m_{n+1}+1} \cdot k^{2m_{n+1}+1} \cdot \left(2^{g_n}\right)^{-\frac{k}{s}(2s-\frac{1}{2}k(k+1)+\eta)} \\
& \cdot \left(2^{G_{n+1}} \right)^{2(k-1)-\frac{m_{n+1}}{s}(2s-\frac{1}{2}k(k+1)+\eta)} \left( X^{\theta} \right)^{\frac{1}{2}k(k+1)\frac{s-k}{s}(b_n-a_n)} \left(X^{-\eta \theta} \right)^{\frac{k}{s}a_n+\frac{s-k}{s}b_n}.
\end{aligned}
\label{eq:Error}
\end{equation}
In the first sum we further make use of Proposition \ref{prop:extraction}, which gives
$$\begin{aligned}
\llbracket K_{a_n,b_n;m_{n+1}}^{g_n,g_{n+1}}(X) \rrbracket \le & C_{n+1}' \cdot \left( 2^{g_n} \right)^{-k+\frac{1}{2}k(k-1)} \left( 2^{g_{n+1}} \right)^{-(2s-\frac{1}{2}k(k+1)+\eta)\left(\frac{m_{n+1}}{s}-\frac{k^2}{s(s-k)}\right)} \\
& \cdot \llbracket I_{a_{n+1},b_{n+1}}^{g_{n+1}}(X) \rrbracket^{\frac{k}{s-k}} \cdot \left(X^{-\eta \theta} \right)^{\frac{s-2k}{s-k} b_n},
\end{aligned}$$
where
\begin{equation}
C_{n+1}'= 2^{\frac{1}{2}k(k+1)+5} \cdot e^{\frac{1}{4}k(3k-2)} \cdot k^{-\frac{1}{2}k(k-2)+1} \cdot (s-k+k^2)^k.
\label{eq:Cprime}
\end{equation}
In the case of $n=0$ we arrive at the inequality
\begin{equation}\begin{aligned}
\Psi_{0} \le & C_{1}C_1'\CC_0G_1 \left(X^{-\eta \theta}\right)^{\frac{s-2k}{s-k}b_0} \!\! \! \! \sum_{g_0= \lceil \log_2(2k) \rceil}^{G_0} \! \left(2^{g_0}\right)^{\alpha_0-k+\frac{1}{2}k(k-1)} \!\! \! \! \sum_{g_1= \lceil \log_2(2k) \rceil}^{G_1} \! \left(2^{g_1}\right)^{\alpha_1} \llbracket I_{a_{1},b_{1}}^{g_{1}}(X) \rrbracket^{\frac{k}{s-k}} \\
& + \CC_0 \sum_{g_0= \lceil \log_2(2k) \rceil}^{G_0} \left(2^{g_0}\right)^{\alpha_0} G_1E_{1}  + \CC_0^{\dagger} \cdot X^{-\eta \theta \frac{s-2k}{s-k}}.
\label{eq:mediumstep0}
\end{aligned}\end{equation}
For $n\ge 1$ we further use the elementary inequality $(x+y)^r \le x^r+y^r$ twice, which holds for $x,y \ge 0$ and $0\le r \le 1$ and arrive at the inequality
\begin{equation} \begin{aligned}
\Psi_{n} \le & \CC_n \left(X^{-\eta \theta} \right)^{\frac{s-2k}{s-k} \sum_{i=0}^{n-1}b_i \left( \frac{k}{s-k} \right)^{i} } \Biggl[ \left(X^{-\eta \theta}\right)^{\frac{s-2k}{s-k}b_n\left(\frac{k}{s-k}\right)} \! \! \! \! \sum_{g_n=\lceil \log_2(2k) \rceil}^{G_n} \left(2^{g_n} \right)^{\alpha_n+\frac{k}{s-k}\left(-k+\frac{1}{2}k(k-1)\right)} \\
& \cdot \Biggl(C_{n+1}C_{n+1}'G_{n+1}\sum_{g_{n+1} = \lceil \log_2(2k) \rceil}^{G_{n+1}} \left(2^{g_{n+1}} \right)^{\alpha_{n+1}} \llbracket I_{a_{n+1},b_{n+1}}^{g_{n+1}}(X) \rrbracket^{\frac{k}{s-k}} \Biggr)^{\frac{k}{s-k}} \\
& + \sum_{g_n= \lceil \log_2(2k) \rceil}^{G_n} \left(2^{g_n} \right)^{\alpha_n}\left(G_{n+1}E_{n+1}\right)^{\frac{k}{s-k}} \Biggr]^{\left( \frac{k}{s-k} \right)^{n-1}} + \CC_n^{\dagger} X^{-\eta \theta \frac{s-2k}{s-k}} \\
\le & \CC_n \left(X^{-\eta \theta} \right)^{\frac{s-2k}{s-k} \sum_{i=0}^{n}b_i \left( \frac{k}{s-k} \right)^{i} } \Biggl[\sum_{g_n=\lceil \log_2(2k) \rceil}^{G_n} \left(2^{g_n} \right)^{\alpha_n+\frac{k}{s-k}\left(-k+\frac{1}{2}k(k-1)\right)} \\
& \cdot \Biggl(C_{n+1}C_{n+1}'G_{n+1}\sum_{g_{n+1} = \lceil \log_2(2k) \rceil}^{G_{n+1}} \left(2^{g_{n+1}} \right)^{\alpha_{n+1}} \llbracket I_{a_{n+1},b_{n+1}}^{g_{n+1}}(X) \rrbracket^{\frac{k}{s-k}} \Biggr)^{\frac{k}{s-k}} \Biggr]^{\left(\frac{k}{s-k}\right)^{n-1}} \\
& + \CC_n\left(X^{-\eta \theta} \right)^{\frac{s-2k}{s-k}} \Biggl[ \sum_{g_n= \lceil \log_2(2k) \rceil}^{G_n} \left(2^{g_n} \right)^{\alpha_n}\left(G_{n+1}E_{n+1}\right)^{\frac{k}{s-k}} \Biggr]^{\left( \frac{k}{s-k} \right)^{n-1}} + \CC_n^{\dagger} X^{-\eta \theta \frac{s-2k}{s-k}}.
\label{eq:mediumstep1}
\end{aligned}\end{equation}

Next we show that the exponent of $2^{g_{n}}$ is at most $-1$ if $n=0$ and $-\frac{1}{3}$ otherwise. First consider the case $n=0$. There we have
$$
2(k-1)-2(m_0-k)-k+\frac{1}{2}k(k-1) \le -1 \Leftrightarrow \frac{1}{4}k(k+1)+k-\frac{1}{2} \le m_0,
$$
which is true. Now we analyse the case when $n>0$. There we have to bound
$$
2(k-1)-\left(2s-\frac{1}{2}k(k+1)+\eta\right)\left(\frac{m_n}{s}-\frac{k^2}{s(s-k)}\right) +\frac{k}{s-k}\left[-k+\frac{1}{2}k(k-1)\right].
$$
Since $m_n \ge 1 \ge \frac{k^2}{s-k}$ we only make the expression bigger when replacing $2s-\frac{1}{2}k(k+1)+\eta$ by $\frac{3}{2}s$ as the latter is smaller. Thus we are left to bound
$$\begin{aligned}
&2(k-1)-\frac{3}{2}\left(m_n-\frac{k^2}{s-k}\right) + \frac{k}{s-k}\left[-k+\frac{1}{2}k(k-1) \right] \\
= & 2(k-1)-\frac{3}{2}m_n + \frac{k}{s-k} \left[\frac{3}{2}k-k+\frac{1}{2}k(k-1) \right]. 
\end{aligned}$$
Now we have $\frac{3}{2}k-k+\frac{1}{2}k(k-1) \ge 0$ and hence the expression gets bigger when we replace $s$ by $k^2+k$ as the latter is smaller. We are left to deal with
$$
\frac{5}{2}k-2-\frac{3}{2}m_n.
$$
It suffices to have
$$
\frac{5}{2}k-2-\frac{3}{2}m_n \le -\frac{1}{3} \Leftrightarrow \frac{5}{3}k-\frac{10}{9} \le m_n,
$$
which is true. Therefore we conclude
\begin{equation}\begin{aligned}
\sum_{g_0= \lceil \log_2(2k) \rceil}^{G_0} \left( 2^{g_0} \right)^{\alpha_0-k+\frac{1}{2}k(k-1)} &\le \frac{1}{k}, \\
\sum_{g_n= \lceil \log_2(2k) \rceil}^{G_n} \left( 2^{g_n} \right)^{\alpha_n+\frac{k}{s-k}\left(-k+\frac{1}{2}k(k-1)\right)} &\le \frac{4}{k^{\frac{1}{3}}} \quad \forall n \ge 1.
\end{aligned}\label{eq:gexpbound}
\end{equation}

Now we turn our attention to the analysis of the error term; i.e. the terms involving $E_{n+1}$. Let us consider the exponent of $2^{g_n}$ first. For $n=0$ the exponent is
$$\begin{aligned}
2(k-1)-2(m_0-k)-\frac{k}{s}\left(2s-\frac{1}{2}k(k+1)+\eta \right) &\le 2(k-1)-2(m_0-k)-\frac{3}{2}k \\
& \le -\frac{1}{2}k^2-\frac{2}{3}k+\frac{2}{3} \le -1.
\end{aligned}$$
Thus we have
\begin{equation}
\sum_{g_0= \lceil \log_{2}(2k) \rceil}^{G_0} \left(2^{g_0}\right)^{\alpha_0-\frac{k}{s}\left(2s-\frac{1}{2}k(k+1)+\eta\right)} \le 2 (2k)^{-\frac{1}{2}k^2-\frac{2}{3}k+\frac{2}{3}}.
\label{eq:errorgsum0}
\end{equation}
For $n \ge 1$ the exponent is
$$\begin{aligned}
2(k-1)-\left(2s-\frac{1}{2}k(k+1)+\eta\right)&\left(\frac{m_n}{s}-\frac{k^2}{s(s-k)}\right)-\frac{k}{s-k} \frac{k}{s}\left(2s-\frac{1}{2}k(k+1)+\eta \right) \\
&= 2(k-1)-\frac{m_n}{s}\left(2s-\frac{1}{2}k(k+1)+\eta\right) \\
&\le 2(k-1)-\frac{3}{2}m_n \\
&\le -\frac{1}{2}k-1.
\end{aligned}$$
Thus we get for $n \ge 1$
\begin{equation}
\sum_{g_n= \lceil \log_{2}(2k) \rceil}^{G_n} \left(2^{g_n}\right)^{\alpha_n-\frac{k}{s-k}\frac{k}{s}\left(2s-\frac{1}{2}k(k+1)+\eta\right)} \le 2 (2k)^{-\frac{1}{2}k-1}.
\label{eq:errorgsum}
\end{equation}

Now we consider the power of $X$ in the error term $E_{n+1}$; i.e. we are having a detailed look at
$$
\left(2^{G_{n+1}} \right)^{2(k-1)-\frac{m_{n+1}}{s}(2s-\frac{1}{2}k(k+1)+\eta)} \left( X^{\theta} \right)^{\frac{1}{2}k(k+1)\frac{s-k}{s}(b_n-a_n)} \left(X^{-\eta \theta} \right)^{\frac{k}{s}a_n+\frac{s-k}{s}b_n}.
$$
For $n=0$ we bound
$$
\left( X^{-\eta \theta} \right)^{\frac{k}{s}a_0+\frac{s-k}{s}b_0} \le X^{-\eta \theta \frac{s-2k}{s-k}}
$$
and for $n \ge 1$ we bound trivially
$$
\left( X^{-\eta \theta} \right)^{\frac{k}{s}a_n+\frac{s-k}{s}b_n} \le 1.
$$
For the rest we use the inequality $G_{n+1} \ge k^{n+1} \theta \log_2(X)-1$ and find
$$\begin{aligned}
& \left(2^{G_{n+1}} \right)^{2(k-1)-\frac{m_{n+1}}{s}(2s-\frac{1}{2}k(k+1)+\eta)} \left( X^{\theta} \right)^{\frac{1}{2}k(k+1)\frac{s-k}{s}(b_n-a_n)} \\
& \qquad \qquad \le  \left( \frac{X^{\theta k^{n+1}}}{2} \right)^{2(k-1)-\frac{3}{2}m_{n+1}} \left( X^{\theta k^{n+1}} \right)^{\frac{1}{2}(k+1)} \\
& \qquad \qquad \le  2^{\frac{1}{2}k+1} \left(X^{\theta k^{n+1}} \right)^{-\frac{1}{2}} \\
& \qquad \qquad \le  2^{\frac{1}{2}k+1} 2^{-\frac{1}{2}k^{n+1}} \\
& \qquad \qquad \le  \begin{cases}2, & n=0, \\ 1, & n\ge 1. \end{cases}
\end{aligned}$$
The latter seems inefficient, but one has to consider that the $(\lambda k)^n$-th root will be taken of it in due course. Hence we have
\begin{equation}\begin{aligned}
\left(2^{G_{n+1}} \right)^{2(k-1)-\frac{m_{n+1}}{s}(2s-\frac{1}{2}k(k+1)+\eta)} & \left( X^{\theta} \right)^{\frac{1}{2}k(k+1)\frac{s-k}{s}(b_n-a_n)} \left(X^{-\eta \theta} \right)^{\frac{k}{s}a_n+\frac{s-k}{s}b_n} \\
&\le \begin{cases}2 X^{-\eta\theta \frac{s-2k}{s-k}}, & n=0, \\ 1, & n\ge 1. \end{cases}
\label{eq:powerX}
\end{aligned}\end{equation}

Lastly we have
\begin{equation}
G_{n+1} \le k^{n-D} \log_2(2X) \le \begin{cases} k^{-D} \log_2(2X), & n=0, \\ k^{-1} \log_2(2X), & n \ge 1. \end{cases}
\label{eq:Gineq}
\end{equation}
By collecting all of the previous analysis we have proven \eqref{eq:ineqpsi}. We go through this one step at a time. For $n=0$ we combine \eqref{eq:mediumstep0} with \eqref{eq:gexpbound} and \eqref{eq:Gineq}; this gives us the main term and $\CC_1$ as in \eqref{eq:c0}. For the error term we combine \eqref{eq:mediumstep0} with \eqref{eq:Error}, \eqref{eq:errorgsum0}, \eqref{eq:powerX} and \eqref{eq:Gineq}. Which gives us $\CC_1^{\dagger}$ as follows:
\begin{equation}
\begin{aligned}
\CC_{1} &= \CC_0 \cdot C_1C_1' \cdot k^{-D} \cdot k^{-1}, \\
\CC_1^{\dagger} &= \CC_0^{\dagger}+  \CC_0 \cdot \! 2^{2m_{1}+1} \cdot \! k^{2m_{1}+1} \cdot \!  2(2k)^{-\frac{1}{2}k^2-\frac{2}{3}k+\frac{2}{3}}\cdot \! k^{-D}\cdot \! 2\\
& \le \CC_0^{\dagger}+ \CC_0 \cdot \! 2^{\frac{10}{3}k+1} \cdot \! k^{\frac{10}{3}k+1} \cdot \!  2(2k)^{-\frac{1}{2}k^2-\frac{2}{3}k+\frac{2}{3}} \cdot \! k^{-D} \cdot \! 2 \\
& \le \CC_0^{\dagger}+ \CC_0 \cdot \!  2^{-\frac{1}{2}k^2+\frac{8}{3}k+\frac{11}{3}}k^{-\frac{1}{2}k^2+\frac{8}{3}k+\frac{5}{3}-D}.
\end{aligned}
\label{eq:c0}
\end{equation}
For $n=1,\dots,D-1$ we combine \eqref{eq:mediumstep1} with \eqref{eq:gexpbound} and \eqref{eq:Gineq}; this gives us the main term with $\CC_{n+1}$ as in \eqref{eq:cn}. For the error term we combine \eqref{eq:mediumstep1} with \eqref{eq:Error}, \eqref{eq:errorgsum}, \eqref{eq:powerX} and \eqref{eq:Gineq}. Which gives us $\CC_{n+1}^{\dagger}$ as follows:
\begin{equation}
\begin{aligned}
\CC_{n+1} &= \CC_n \cdot \left(\frac{4}{k^{\frac{1}{3}}} \right)^{\left(\frac{k}{s-k}\right)^{n-1}} \left(C_{n+1}C_{n+1}' \cdot k^{-1} \right)^{\left(\frac{k}{s-k}\right)^n}, \\
\CC_{n+1}^{\dagger} &= \CC_n^{\dagger}+\CC_n \left( 2(2k)^{-\frac{1}{2}k-1} \right)^{\left(\frac{k}{s-k} \right)^{n-1}} \left( 2^{2m_{n+1}+1} k^{2m_{n+1}+1} k^{-1} \right)^{\left(\frac{k}{s-k} \right)^{n}}.
\end{aligned}
\label{eq:cn}
\end{equation}
We bound $\CC^{\dagger}_{n+1}$ further by
\begin{equation}\begin{aligned}
\CC_{n+1}^{\dagger} &\le \CC_{n}^{\dagger} + \CC_n \left( \left( 2(2k)^{-\frac{1}{2}k-1} \right)^k \cdot 2^{2m_{n+1}+1} k^{2m_{n+1}+1} k^{-1} \right)^{\left(  \frac{k}{s-k}\right)^n}\\
&\le  \CC_{n}^{\dagger} + \CC_n \left(  2^{-\frac{1}{2}k^2+\frac{10}{3}k+1}   k^{-\frac{1}{2}k^2+\frac{7}{3}k}  \right)^{\left(  \frac{k}{s-k}\right)^n}.
\label{eq:cndagger}
\end{aligned}\end{equation}
\end{proof}

It remains to estimate $\Psi_{D}$. This is done with the help of Proposition \ref{prop:endI} and yields the following proposition.

\begin{prop} With the assumptions as in Proposition \ref{prop:ineqpsi} we have
$$
\Psi_{D} \le \CC_{D+1} X^{\frac{k^2(k^2-1)}{2s} b_D \left(\frac{k}{s-k} \right)^D \theta-\eta\frac{s-2k}{s-k} \sum_{i=0}^{D}b_i \left( \frac{k}{s-k} \right)^{i} \theta } + \CC_D^{\dagger} \cdot X^{-\eta \theta\frac{s-2k}{s-k}},
$$
where
$$
\CC_{D+1} = \CC_D \cdot \left( \frac{4}{k^{\frac{1}{3}}} \right)^{\left(\frac{k}{s-k} \right)^{D-1}} {C_{D+1}'}^{\left( \frac{k}{s-k} \right)^D}
$$
and
$$\begin{aligned}
C_{D+1}' &= 2^{\frac{1}{2}k(k+5)+4} \cdot e^{\frac{1}{4}k(3k-2)} \cdot k^{-\frac{1}{2}k(k-2)+1} \cdot (s-k+k^2)^k \\
&= C_{D}' \cdot 2^{2k} .
\end{aligned}$$
\label{prop:lastpsi}
\end{prop}
\begin{proof} The proof follows from Proposition \ref{prop:endI} combined with \eqref{eq:gexpbound} applied to \eqref{eq:psidef}.
\end{proof}

We are left with estimating the constants. For $n = 1,\dots,D$ we have from \eqref{eq:Cdefi}
$$\begin{aligned}
C_n & = 2^{4m_{n}-2k+2} \cdot k^{2m_{n}} \le 2^{\frac{14}{3}k+2} \cdot k^{\frac{10}{3}k}
\end{aligned}$$
and from \eqref{eq:Cprime}
$$\begin{aligned}
C_n' &= 2^{\frac{1}{2}k(k+1)+5} \cdot e^{\frac{1}{4}k(3k-2)} \cdot k^{-\frac{1}{2}k(k-2)+1} \cdot (s-k+k^2)^k.
\end{aligned}$$
Inserting this into Proposition \ref{prop:ineqpsi} using the Definition \eqref{eq:cn} we get\footnote{Here and throughout this section $\sum_{i=0}^{-l}$ denotes the empty sum for any $l > 0$ and equals $0$.}
\begin{equation}\begin{aligned}
\CC_n \le& 2^{\frac{3}{2}k^2+\frac{11}{2}k-1}k^{\frac{1}{2}k^2+\frac{19}{6}k-2-2D} \left(1+\frac{1}{X^{\theta}} \right)^{2(s-m_0)} \cdot \left( \frac{4}{k^{\frac{1}{3}}} \right)^{\sum_{i=0}^{n-2} \left(\frac{k}{s-k} \right)^{i}} \\
&\cdot \left( 2^{\frac{1}{2}k^2+\frac{31}{6}k+7}e^{\frac{3}{4}k^2-\frac{1}{2}k}k^{-\frac{1}{2}k^2+\frac{19}{3}k}(\lambda+1)^k \right)^{\sum_{i=0}^{n-1} \left(\frac{k}{s-k} \right)^{i}} 
\label{eq:cbound}
\end{aligned}\end{equation}
for $n=1,\dots,D$. We further have
$$
2^{2k} \le 2^{\frac{14}{3}k+2} \cdot k^{\frac{10}{2}k} \cdot k^{-1}.
$$
Using these two inequalities with Proposition \ref{prop:lastpsi} we get
\begin{equation}\begin{aligned}
\CC_{D+1} \le& 2^{\frac{3}{2}k^2+\frac{11}{2}k-1}k^{\frac{1}{2}k^2+\frac{19}{6}k-2-2D} \left(1+\frac{1}{X^{\theta}} \right)^{2(s-m_0)} \cdot \left( \frac{4}{k^{\frac{1}{3}}} \right)^{\sum_{i=0}^{D-1} \left(\frac{k}{s-k} \right)^{i}} \\
&\cdot \left( 2^{\frac{1}{2}k^2+\frac{31}{6}k+7}e^{\frac{3}{4}k^2-\frac{1}{2}k}k^{-\frac{1}{2}k^2+\frac{19}{3}k}(\lambda+1)^k \right)^{\sum_{i=0}^{D} \left(\frac{k}{s-k} \right)^{i}}.
\label{eq:cboundlast}
\end{aligned}\end{equation}

We now turn our attention to bounding $\CC^{\dagger}_n$. We continue the estimation \eqref{eq:c0} for $\CC_1^{\dagger}$:
\begin{equation}\begin{aligned}
\CC^{\dagger}_1 \le & 
2^{k^2+\frac{11}{3}k}k^{\frac{1}{2}k^2+\frac{19}{6}k-2-D} + 2^{\frac{3}{2}k^2+\frac{11}{2}k-1}k^{\frac{1}{2}k^2+\frac{19}{6}k-2-D} \left(1+\frac{1}{X^{\theta}} \right)^{2(s-m_0)} \\
& \cdot 2^{-\frac{1}{2}k^2+\frac{8}{3}k+\frac{11}{3}}k^{-\frac{1}{2}k^2+\frac{8}{3}k+\frac{5}{3}-D} \\
\le & 2^{\frac{3}{2}k^2+\frac{11}{2}k-1}k^{\frac{1}{2}k^2+\frac{19}{6}k-2-D} \left(1+\frac{1}{X^{\theta}} \right)^{2(s-m_0)} \\
& \cdot \left(2^{-\frac{1}{2}k^2-\frac{11}{6}k+1}+2^{-\frac{1}{2}k^2+\frac{8}{3}k+\frac{11}{3}}k^{-\frac{1}{2}k^2+\frac{8}{3}k+\frac{5}{3}-D} \right).
\label{eq:Cdag1}
\end{aligned}\end{equation}
Using induction on \eqref{eq:cndagger} with \eqref{eq:cbound} and \eqref{eq:Cdag1} as base we further find
\begin{equation}\begin{aligned}
\CC^{\dagger}_n \le& 2^{\frac{3}{2}k^2+\frac{11}{2}k-1}k^{\frac{1}{2}k^2+\frac{19}{6}k-2-D} \left(1+\frac{1}{X^{\theta}} \right)^{2(s-m_0)} \\ & \cdot \Biggl[2^{-\frac{1}{2}k^2-\frac{11}{6}k+1} + 2^{-\frac{1}{2}k^2+\frac{8}{3}k+\frac{11}{3}}k^{-\frac{1}{2}k^2+\frac{8}{3}k+\frac{2}{3}} \\
& + k^{-1}\sum_{i=2}^{n} \Biggl ( \left(\frac{4}{k^{\frac{1}{3}}}\right)^{\sum_{j=0}^{i-3}\left( \frac{k}{s-k} \right)^j}  \left( 2^{\frac{1}{2}k^2+\frac{31}{6}k+7}e^{\frac{3}{4}k^2-\frac{1}{2}k}k^{-\frac{1}{2}k^2+\frac{19}{3}k}(\lambda+1)^k  \right)^{\sum_{j=0}^{i-1}\left(\frac{k}{s-k}\right)^j} \Biggr) \Biggr]
\label{eq:cdaggen}
\end{aligned}\end{equation}
for $n=1,\dots,D$, where we have made use of the inequality
$$\begin{aligned}
2^{-\frac{1}{2}k^2+\frac{10}{3}k+1}   k^{-\frac{1}{2}k^2+\frac{7}{3}k} 
\le 2^{\frac{1}{2}k^2+\frac{31}{6}k+7}e^{\frac{3}{4}k^2-\frac{1}{2}k}k^{-\frac{1}{2}k^2+\frac{19}{3}k}(\lambda+1)^k
\end{aligned}$$
and $k^{-D}\le k^{-1}$. Let us now tame the inequality \eqref{eq:cdaggen}. We have for any $n \in \BZ$
$$
\left( \frac{4}{k^{\frac{1}{3}}} \right)^{\sum_{i=0}^n \left( \frac{k}{s-k} \right)^i} \le \max \left\{ 1, \sup_{k \ge 3} \left( \frac{4}{k^{\frac{1}{3}}} \right)^{\frac{k}{k-1}}  \right\} \le \frac{2^3}{\sqrt{3}} \le 2k
$$
as the latter is a decreasing function in $k$. Let $\CM$ denote the maximum of the quantities
\begin{equation}\begin{aligned}
2^{-\frac{1}{2}k^2-\frac{11}{6}k}&, \\
\cdot 2^{-\frac{1}{2}k^2+\frac{8}{3}k+\frac{8}{3}}k^{-\frac{1}{2}k^2+\frac{8}{3}k+\frac{2}{3}}&, \\
\left( 2^{\frac{1}{2}k^2+\frac{31}{6}k+7}e^{\frac{3}{4}k^2-\frac{1}{2}k}k^{-\frac{1}{2}k^2+\frac{19}{3}k}(\lambda+1)^k \right)^{\gamma}&, \quad \gamma\in \left\{1, \frac{s-k}{s-2k}\right\},
\label{eq:mdef}
\end{aligned}\end{equation}
then we have
\begin{equation*} \begin{aligned}
\CC_D^{\dagger} &\le  2^{\frac{3}{2}k^2+\frac{11}{2}k}k^{\frac{1}{2}k^2+\frac{19}{6}k-2-D} \left(1+\frac{1}{X^{\theta}} \right)^{2(s-m_0)}  \cdot (D+1) \CM \\
& \le 2^{\frac{3}{2}k^2+\frac{11}{2}k}k^{\frac{1}{2}k^2+\frac{19}{6}k-2} \left(1+\frac{1}{X^{\theta}} \right)^{2(s-m_0)}  \CM.
\end{aligned}\end{equation*}
Returning to \eqref{eq:cboundlast}, we also have
\begin{equation*}
\CC_{D+1} \le 2^{\frac{3}{2}k^2+\frac{11}{2}k}k^{\frac{1}{2}k^2+\frac{19}{6}k-2} \left(1+\frac{1}{X^{\theta}} \right)^{2(s-m_0)}  \CM.
\end{equation*}
We immediately see that the middle expression in \eqref{eq:mdef} is dominated by the latter one. We also make use of the inequality $\lambda+1\le k^2$ and hence $\CM$ is at most $\CM_0$, where we recall \eqref{eq:M0}:
\begin{equation*}\begin{aligned}
\CM_0 =  \max_{\gamma \in \{1,\frac{s-k}{s-2k}\}} \Biggl\{  \left( 2^{\frac{1}{2}k^2+\frac{31}{6}k+7}e^{\frac{3}{4}k^2-\frac{1}{2}k}k^{-\frac{1}{2}k^2+\frac{25}{3}k} \right)^{\gamma}, 2^{-\frac{1}{2}k^2-\frac{11}{6}k} \Biggr \}.
\end{aligned}\end{equation*}
We conclude the following proposition.
\begin{prop}
Let $s,k \in \BN$ with $k \ge 3$ and $2 \log(k) \ge \lambda=\frac{s-k}{k^2}\ge 1$. Further let $D \ge 1$ be an integer and set $\theta=k^{-(D+1)}$. Assume that
$$
J_{s,k}(X) \le C \log_2(2X)^{\delta} X^{2s-\frac{1}{2}k(k+1)+\eta} \quad \forall X \ge 1,
$$
for some $0 \le \delta$ and $0 < \eta \le \frac{1}{2}k(k+1)$. Then we have
$$
J_{s,k}(X) \le C' \log_2(2X)^{\delta+\frac{2\lambda k-1}{\lambda k-1}} X^{2s-\frac{1}{2}k(k+1)+\eta} \left(X^{\Delta \theta}+ X^{-\eta \theta \frac{s-2k}{s-k}} \right) \quad \forall X^{\theta} \ge 2,
$$
where
$$
\Delta=\frac{k^2(k^2-1)}{2s}\lambda^{-D}-\eta\frac{s-2k}{s-k}\sum_{i=0}^D\lambda^{-i},
$$
$$
C'=C \cdot 2^{\frac{3}{2}k^2+\frac{11}{2}k}k^{\frac{1}{2}k^2+\frac{19}{6}k-2} \left(1+\frac{1}{X^{\theta}} \right)^{2(s-m_0)}  \CM_0,
$$
with $m_0$ as in \eqref{eq:m} and where $\CM_0$ is defined as in \eqref{eq:M0}.
\label{prop:oneit}
\end{prop}

Now we want to bring Proposition \ref{prop:oneit} into a shape which one can iterate easily. For this matter we want to optimise our gain in the exponent. The optimal choice of $D$ is in general not an easy problem and leads to complications in further calculations. Nevertheless there is a reasonable exponent gain one can achieve, namely $-\eta \theta \frac{s-2k}{s-k}$. This is reasonable because if $\lambda$ is close to $1$ all terms are of almost equal size and if $\lambda$ is large the positive term gets very small and thus can be handled by the tail sum.

Let us first assume $\lambda >1$, then
$$
\Delta = -\eta\frac{s-2k}{s-k}+\frac{k^2(k^2-1)}{2s}\lambda^{-D}-\eta \frac{s-2k}{s-k} \frac{\lambda^{-1}(1-\lambda^{-D})}{1-\lambda^{-1}},
$$
and we would like $\Delta \le -\eta \frac{s-2k}{s-k}$. Thus we need
$$
\left(\frac{k^2(k^2-1)}{2s}+\eta\frac{s-2k}{s-k}\frac{\lambda^{-1}}{1-\lambda^{-1}}\right)\lambda^{-D} \le \eta \frac{s-2k}{s-k}\frac{\lambda^{-1}}{1-\lambda^{-1}}
$$
or
$$
\frac{k^2}{2\eta} \frac{k^2-1}{s} \frac{s-k}{s-2k}\frac{\lambda-1}{\lambda} +1\le \lambda^D.
$$
Now we have
$$
\frac{k^2-1}{s} \frac{s-k}{s-2k} \frac{\lambda-1}{\lambda}=\frac{k^2-1}{\lambda k+1} \frac{\lambda}{\lambda k -1 } \frac{\lambda-1}{\lambda} = \frac{(\lambda-1)(k^2-1)}{\lambda^2 k^2-1} \le \frac{\lambda-1}{\lambda^2},
$$
hence it suffices to have
\begin{equation}
D \ge \frac{\log\left(\frac{k^2}{2\eta} \frac{\lambda-1}{\lambda^2}+1 \right)}{\log(\lambda)}.
\label{eq:Dmin}
\end{equation}
In the case $\lambda =1$ one needs
$$
\frac{k^2(k^2-1)}{2s}-\eta \frac{s-2k}{s-k}D \le 0 \Leftrightarrow D \ge \frac{k^2}{2\eta},
$$
which is recovered from \eqref{eq:Dmin} in the limit as $\lambda \to 1^{+}$.\\

We are now able to balance the two inequalities in Proposition \ref{prop:oneit}. We make the choice $X_0^{\theta}=4k$ and use the trivial inequality for $1 \le X \le X_0$ and the new inequality for $X\ge X_0$. Thus we have for $1 \le X \le X_0$:
\begin{equation}\begin{aligned}
J_{s,k}(X) &\le C \log_2(2X)^{\delta} X^{2s-\frac{1}{2}k(k+1)+\eta} \\
& \le C \log_2(2X)^{\delta} X^{2s-\frac{1}{2}k(k+1)+\eta} \left( X_{0}^{\eta \theta} \cdot  X^{-\eta \theta \frac{s-2k}{s-k}} \right) \\
& \le C \cdot 2^{k^2+k}k^{\frac{1}{2}k^2+\frac{1}{2}k} \cdot  \log_2(2X)^{\delta+ \frac{2 \lambda k-1}{\lambda k -1}} X^{2s-\frac{1}{2}k(k+1)+\eta} \cdot  X^{-\eta \theta \frac{s-2k}{s-k}},
\end{aligned}\label{eq:Xsmall}
\end{equation}
where we have made use of $\eta \le \frac{1}{2}k(k+1)$. For $X \ge X_0$ we further need to estimate
$$
\left(1+\frac{1}{X_0^{\theta}} \right)^{2(s-m_0)} \le \left(1+\frac{1}{4k} \right)^{2\lambda k^2} \le e^{\frac{1}{2}\lambda k} \le k^k.
$$
Thus in this case we get
\begin{equation}\begin{aligned}
J_{s,k}(X) \le & C \cdot 2^{\frac{3}{2}k^2+\frac{11}{2}k+1}k^{\frac{1}{2}k^2+\frac{25}{6}k-2}\cdot \CM_0 \\& \cdot \log_2(2X)^{\delta+ \frac{2 \lambda k -1 }{\lambda k -1}} X^{2s-\frac{1}{2}k(k+1)+\eta} \cdot X^{-\eta \theta \frac{s-2k}{s-k}}.
\end{aligned}\label{eq:Xlarge}
\end{equation}
Comparing the two constants in \eqref{eq:Xsmall} and \eqref{eq:Xlarge} we find that the latter is larger and thus we conclude the proof of Theorem \ref{thm:final}.

\section{Final Upper Bounds}

In this section we consider a more general system of equations
\begin{equation*}
\sum_{i=1}^l x_i^j-\sum_{i=l+1}^s x_i^j = N_j, \quad (j=1,\dots,k),
\end{equation*}
with integers $1\le x_i \le X$. Let $I_{s,k,l}(\vect{N};X)$ denote its counting function. We shall use a Hardy--Littlewood dissection into major and minor arcs to establish an asymptotic formula
$$
I_{s,k,l}(\vect{N};X) \sim \FS_{s,k,l}(\vect{N}) \CJ_{s,k,l}(\vect{N}) X^{s-\frac{1}{2}k(k+1)}
$$
with an effective error term, where $\FS_{s,k,l}(\vect{N})$ and $\CJ_{s,k,l}(\vect{N})$ are the singular series and the singular integral, which are given by

$$\FS_{s,k}(\vect{N}) = \sum_{q_1,\dots,q_n=1}^{\infty} \sum_{\substack{\vect{a} \mod \vect{q}\\ (a_i,q_i)=1, i=1,\dots,k}} (q_1\cdot \dots \cdot q_k)^{-s} S_{\vect{q}}(\vect{a})^l \overline{S_{\vect{q}}(\vect{a})}^{s-l} e\left( \sum_{j=1}^k \frac{a_jN_j}{q_j} \right)$$
and
$$\CJ_{s,k}(\vect{N}) = \int_{\BR^k} I(\vect{\beta})^l\overline{I(\vect{\beta})}^{s-l} e\left( \sum_{j=1}^k \frac{\beta_jN_j}{X^j} \right) d\vect{\beta},$$
where
$$S_{\vect{q}}(\vect{a}) = \sum_{n=1}^q e\left( \sum_{j=1}^k \frac{a_jn^j}{q_j} \right) \text{ and } I(\vect{\beta}) = \int_{0}^1 e\left( \sum_{j=1}^k \beta_jx^j \right) dx.$$
We achieve this by using a good enough estimate for $J_{s,k}(X)$ in the minor arcs, which we will get by iterating Theorem \ref{thm:final}. To make our life simpler we restrict to the case $\lambda>1$ and think of $\lambda$ as fixed as in this case we see that $D$ only grows logarithmically in $\frac{k^2}{2\eta}$, which in return makes the constant smaller.

We will iterate Theorem \ref{thm:final} as follows. We fix $D$ and iterate as many times as needed till we get an exponent $\eta$ that is too small to apply the theorem with the choice of $D$ we fixed. For this purpose we need to reverse engineer the inequality \eqref{eq:Dmin}. We have
\begin{equation}\begin{aligned}
\frac{k^2}{2\eta} \cdot \frac{\lambda-1}{\lambda^2}+1 &= \frac{k^2}{2\eta} \cdot \frac{\lambda-1}{\lambda^2} \left(1+ \frac{2\eta}{k^2} \cdot \frac{\lambda^2}{\lambda-1} \right) \\
&\le \frac{k^2}{2\eta} \cdot \frac{\lambda-1}{\lambda^2} \left(1+ \frac{k(k+1)}{k^2} \cdot \frac{\lambda^2}{\lambda-1} \right) \\
&\le \frac{k^2}{2\eta} \cdot \frac{\lambda-1}{\lambda^2} \left(\frac{k(k+1)}{k^2} \cdot \frac{\lambda^2+\lambda-1}{\lambda-1} \right) \\
&= \frac{k(k+1)}{2\eta} \cdot \frac{\lambda^2+\lambda-1}{\lambda^2} \\
&\le \frac{5}{4} \cdot \frac{k(k+1)}{2\eta}
\end{aligned}
\label{eq:senf1}
\end{equation}
as $\lambda \le \frac{1}{4}\lambda^2+1$ by AM-GM. Thus we are able to apply Theorem \ref{thm:final} as long as
$$
\eta \ge \frac{5}{4} \cdot \frac{k(k+1)}{2\lambda^D},
$$
which immediately leads to the following proposition.
\begin{prop} Let $s,k,D \in \BN$ with $k \ge 3$ and $2 \log(k) \ge \lambda = \frac{s-k}{k^2}>1$. Assume that
$$
J_{s,k}(X) \le C \log_2(2X)^{\delta} X^{2s-\frac{1}{2}k(k+1)+\eta} \quad \forall X \ge 1,
$$
for some $0 \le \delta$ and $0 < \eta \le \max \left\{ \frac{1}{2}k(k+1),\frac{5}{4} \cdot \frac{k(k+1)}{2 \lambda^{D-1}} \right\}$. Then we have
$$
\begin{aligned}
J_{s,k}(X) \le & C \left[ 2^{\frac{3}{2}k^2+\frac{11}{2}k+1}k^{\frac{1}{2}k^2+\frac{25}{6}k-2} \CM_0 \! \cdot \! \log_2(2X)^{\frac{2\lambda k -1 }{\lambda k -1}} \right]^{\log(\lambda) \frac{\lambda k}{\lambda k-1}k^{D+1}+1} \\
& \cdot \log_2(2X)^{\delta} X^{2s-\frac{1}{2}k(k+1)+\eta'}, \quad \forall X \ge 1,
\end{aligned}
$$
for some $\eta'<\frac{5}{4}\cdot \frac{k(k+1)}{2 \lambda^{D}}$ and where $\CM_0$ as defined in \eqref{eq:M0}.
\label{prop:OneDit}
\end{prop}
\begin{proof} If $\eta < \frac{5}{4}\cdot \frac{k(k+1)}{2 \lambda^{D}}$, then the statement is trivial. Otherwise we are able to apply Theorem \ref{thm:final} and we receive an inequality with
$$
\eta'=\eta\left(1-\frac{1}{k^{D+1}} \frac{\lambda k -1}{\lambda k} \right).
$$
If $\eta'< \frac{5}{4}\cdot \frac{k(k+1)}{2 \lambda^{D}}$ then we are done otherwise we repeat the process. After at most 
$$
\left\lceil k^{D+1}\frac{\lambda k}{\lambda k-1 } \log(\lambda) \right\rceil  \le k^{D+1}\frac{\lambda k}{\lambda k-1 } \log(\lambda)+1
$$ iterations we are guaranteed to have $\eta'< \frac{5}{4}\cdot \frac{k(k+1)}{2 \lambda^{D}}$ and hence conclude the proof of the proposition.\end{proof}
We get the following corollary immediately.
\begin{cor} Let $s,k,D \in \BN$ with $k \ge 3$ and $2 \log(k) \ge\lambda=\frac{s-k}{k^2}>1$. Then we have
$$
\begin{aligned}
J_{s,k}(X) \le & \left[ 2^{\frac{3}{2}k^2+\frac{11}{2}k+1}k^{\frac{1}{2}k^2+\frac{25}{6}k-2} \CM_0 \! \cdot \! \log_2(2X)^{\frac{2 \lambda k -1 }{\lambda k -1}} \right]^{\!\log(\lambda) \frac{\lambda k}{\lambda k-1}k^{2}\frac{k^D\!-1}{k-1}+D} \\
& \cdot X^{2s-\frac{1}{2}k(k+1)+\frac{5}{4}\cdot\frac{k(k+1)}{2 \lambda^{D}}}, \quad \forall X \ge 1,
\end{aligned}
$$
where $\CM_0$ as defined in \eqref{eq:M0}.
\label{cor:Jupperbound}
\end{cor}

The next step is to get an asymptotic formula as well as an upper bound of the right order of magnitude. From now on we restrict ourselves to the case $\lambda=2$, i.e. $s=2k^2+k$. For this purpose we follow the argument throughout pages 114 to 132  of \cite{ACK04} and insert Corollary \ref{cor:Jupperbound} in the treatment of the minor arcs.

First we bring the estimate in Corollary \ref{cor:Jupperbound} into a shape without logarithms. For $X \ge 7$ we have $\log_2(2X) \le 2 \log(X)$. Moreover we have $\frac{2 k}{2 k -1} \le \frac{6}{5}$. Furthermore we have the inequality
$$
\log(X)^{\alpha} \le \left(\frac{\alpha}{\beta e} \right)^{\alpha} X^{\beta}, \quad \forall \alpha,\beta>0, X\ge e,
$$
as the function $\alpha \log( \log(X)) -\beta \log(X)$ reaches its maximum at $X=e^{\frac{\alpha}{\beta}}$. Hence we conclude for $X \ge 7$ that
\begin{equation*}\begin{aligned}
& \log_2(2X)^{\frac{4 k -1}{2 k -1} \left(\log(2) \frac{2 k}{2 k-1}k^{2}\frac{k^D-1}{k-1}+D\right)} \\
 \le & \left(2\cdot \frac{\frac{4 k -1}{2 k -1} \left(\log(2) \frac{2 k}{2 k-1}k^{2}\frac{k^D-1}{k-1}+D\right)}{\frac{5e}{4} \cdot \frac{k(k+1)}{2^{D+1}}} \right)^{\frac{4 k -1}{2 k -1} \left(\log(2) \frac{2 k}{2 k-1}k^{2}\frac{k^D-1}{k-1}+D\right)} X^{\frac{5}{4} \cdot \frac{k(k+1)}{2^{D+1}}} \\
 \le &  \left[ \frac{2.6\cdot2^{D}}{k(k+1)} \left(\frac{6}{5} \log(2) k^2 \frac{k^{D}-1}{k-1}+D \right)  \right]^{\frac{11}{5}\left(\frac{6}{5}\log(2) k^{2}\frac{k^D-1}{k-1}+D\right)}  X^{\frac{5}{4} \cdot \frac{k(k+1)}{2^{D+1}}}
\end{aligned}\end{equation*}
holds. Furthermore we have
\begin{equation*}\begin{aligned}
\frac{6}{5} \log(2)k^2\frac{k^D-1}{k-1}+D &\le \frac{6}{5} \log(2)k^2\frac{k^D-1}{k-1} \left(1+\frac{D}{\frac{6}{5}\log(2)k^{D+1}} \right) \\
& \le \frac{6}{5} \log(2)k^2\frac{k^D-1}{k-1} \left(1+\frac{5}{54 \log(2)} \right) \\
& \le k^2\frac{k^D-1}{k-1} \\
& \le \frac{3}{2} k^{D+1}
 \end{aligned}
\end{equation*}
and
$$
\frac{2.6}{k(k+1)}k^2\frac{k^D-1}{k-1} \le 2.6 \cdot \frac{k}{k^2-1} \cdot k^D \le k^D.
$$
Hence we may conclude that
\begin{equation*}\begin{aligned}
\log_2(2X)^{\frac{4 k -1}{2 k -1} \left(\log(2) \frac{2 k}{2 k-1}k^{2}\frac{k^D-1}{k-1}+D\right)} \le &  \left(2^D k^{D}   \right)^{\frac{11}{5}\left(\frac{3}{2} k^{D+1}\right)} X^{\frac{5}{4} \cdot \frac{k(k+1)}{2^{D+1}}}. 
\end{aligned}\end{equation*} 
Hence by increasing the constant appropriately we are able to have that the dependency on $X$ is only $X^{2s-\frac{1}{2}k(k+1)+\frac{5}{4} \cdot \frac{k(k+1)}{2^D}}$. We now make use of this inequality in the treatment of $I_2$ in \cite{ACK04} on page 121 with $k_1=k^2$ and $k_2=2k^2+k$. In order to have a power saving we need
\begin{equation*}
\frac{5}{4} \cdot \frac{k(k+1)}{2^D} < k^2 \cdot \rho = k^2\cdot(8k^2(\log(k)+1.5\log(\log(k))+4.2 ))^{-1},
\end{equation*}
which is equivalent to
\begin{equation*}
10 k(k+1)(\log(k)+1.5\log(\log(k))+4.2 ) < 2^{D}.
\end{equation*}
Since we have $k+1 \le \frac{4}{3}k$ and $1.5 \log(\log(k))+4.2 \le 4 \log(k)$ for $k \ge 3$ it is sufficient to have
$$
\frac{200}{3}k^2\log(k) < 2^D
$$
or
\begin{equation}
D= \left \lceil \frac{2 \log(k)+\log(\log(k))+4.2}{\log(2)} \right \rceil \le \frac{2 \log(k)+\log(\log(k))+4.2}{\log(2)}+1.
\label{eq:D}
\end{equation}
Hence we conclude that
\begin{equation}\begin{aligned}
|I_2| \le & \Biggl[ 2^{\frac{3}{2}k^2+\frac{11}{2}k+1+D}k^{\frac{1}{2}k^2+\frac{25}{6}k-2+D}  \CM_0 \Biggr]^{\frac{33}{10} k^{D+1}} \cdot (2k)^{2k^3+11k^2}  \cdot X^{2s-\frac{1}{2}k(k+1)-\delta}
\label{eq:I2}
\end{aligned}\end{equation}
for some $\delta>0$ and where $D$ is given by \eqref{eq:D} and $\CM_0$ as defined in \eqref{eq:M0}. The rest of the calculation goes through as in \cite{ACK04} except that one has to increase the constant to four times the maximum out of $k^{30k^3}$ and the constant in Equation \eqref{eq:I2}. Hence we conclude the following theorem.

\begin{thm} Let $k \ge 3$, $s \ge 5k^2+2k$. Furthermore let $X \ge s^{10}$. We have the asymptotic formula:
$$
\left|I_{s,k,l}(\vect{N};X) - \FS_{s,k,l}(\vect{N})\CJ_{s,k,l}(\vect{N}) X^{s-\frac{1}{2}k(k+1)}\right|  \le C \cdot X^{s-\frac{1}{2}k(k+1)-\delta},
$$
as well as the estimate
$$
I_{s,k,l}(\vect{N};X) \le C X^{s-\frac{1}{2}k(k+1)},
$$
where $C$ is the maximum of $4k^{30k^3}$ and
$$\begin{aligned}
& \Biggl[ 2^{\frac{3}{2}k^2+\frac{11}{2}k+1+D}k^{\frac{1}{2}k^2+\frac{25}{6}k-2+D}  \CM_0 \Biggr]^{\frac{33}{10} k^{D+1}} \cdot 4(2k)^{2k^3+11k^2},
\end{aligned}$$
where $\CM_0$ as defined in \eqref{eq:M0} and
$$
D= \left \lceil \frac{2 \log(k)+\log(\log(k))+4.2}{\log(2)} \right \rceil.
$$
\label{thm:ultimo}
\end{thm}


\bibliography{Bibliography}
\end{document}